\documentclass[12pt,a4paper]{article}%
\usepackage{amssymb}
\usepackage{amsmath}
\usepackage{amsfonts}
\usepackage[dvips]{graphicx}
\usepackage{mathrsfs}
\usepackage[pdftex]{hyperref}
\usepackage{pdfpages}%
\usepackage[all]{xy}
\providecommand{\U}[1]{\protect\rule{.1in}{.1in}}
\hypersetup{
	colorlinks=true,
	citecolor = blue
}
\newtheorem{theorem}{Theorem}

\newtheorem{corollary}[theorem]{Corollary}

\newtheorem{definition}[theorem]{Definition}

\newtheorem{lemma}[theorem]{Lemma}

\newtheorem{proposition}[theorem]{Proposition}
\newtheorem{remark}[theorem]{Remark}

\newenvironment{proof}[1][Proof]{\noindent\textbf{#1.} }{\ \rule{0.5em}{0.5em}}
\begin{document}
	
	\author{V\'{\i}ctor A. Vicente-Ben\'{\i}{}tez\\{\small Instituto de Matemáticas de la U.N.A.M. Campus Juriquilla}\\{\small Boulevard Juriquilla 3001, Juriquilla, Querétaro C.P. 076230 M\'{e}xico }\\{\small  va.vicentebenitez@im.unam.mx, } }
	\title{Bergman spaces for the bicomplex Vekua equation with bounded coefficients}
	\date{}
	\maketitle
\begin{abstract}
We develop the theory for the Bergman spaces of generalized $L_p$-solutions of the bicomplex-Vekua equation $\overline{\boldsymbol{\partial}}W=aW+b\overline{W}$ on bounded domains, where the coefficients $a$ and $b$ are bounded bicomplex-valued functions. We study the completeness of the Bergman space, the regularity of the solutions and the boundedness of the evaluation functional. For the case $p=2$, the existence of a reproducing kernel is established, along with a representation of the orthogonal projection onto the Bergman space in terms of the obtained reproducing kernel, and an explicit expression for the orthogonal complement. Additionally, we analyze the main Vekua equation ($a=0$, $b = \frac{\overline{\boldsymbol{\partial}}f}{f}$ with $f$ being a non-vanishing complex-valued function). Results concerning its relationship with a pair of conductivity equations, the construction of metaharmonic conjugates, and the Runge property are presented.
\end{abstract}


\section{Introduction}
In this work, we study the bicomplex Vekua equation
\begin{equation}\label{eqVekuaintro}
	\overline{\boldsymbol{\partial}}W(z)-a(z)W(z)-b(z)\overline{W}(z)=0, \quad z\in \Omega,
\end{equation}
where $\Omega\subset \mathbb{C}$ is a bounded domain, $\overline{\boldsymbol{\partial}}:=\frac{1}{2}\left(\frac{\partial}{\partial_x}+\mathbf{j}\frac{\partial}{\partial y}\right)$ is the bicomplex Cauchy-Riemann operator, and $a$ and $b$ are essentially bounded functions in $\Omega$. The conjugation $\overline{W}$ is with respect to the bicomplex unit $\mathbf{j}$. Bicomplex equations of the form \eqref{eqVekuaintro} appear, for example, in the study of the factorization of the Schr\"odinger equation with a complex potential \cite{CamposBicomplex,hugo,kravpseudo1,pseudoanalyticvlad}, and in the study of the Dirac system with fixed energy and with scalar potential in the two-dimensional case \cite{camposmendez,castanieda}. The study of classical solutions of Eq. \eqref{eqVekuaintro} in the complex case (where the considered  coefficients  and solutions are complex-valued) is intrinsically related to the theory of pseudoanalytic functions \cite{bers,Colton1}, the study of the Schr\"odinger and conductivity equations \cite{leblont,hugoraul,delgadoleblond,kravpseudo1}, and Calderon's inverse problem in the plane \cite{aslatalapaivarinta}. Recently, the theory of pseudoanalytic functions has been generalized to the bicomplex case \cite{CamposBicomplex}, and for the case of the {\it main Vekua equation} (when $a=0$ and $b=\frac{\overline{\boldsymbol{\partial}}f}{f}$ with $f$ being a non-vanishing complex-valued function), its applications include the solution of elliptic complex equations \cite{hugo,rochon} and the explicit construction of complete systems of solutions \cite{berglez,camposmendez,hugo,minevekua1}. In \cite{delgadoleblond}, the Bergman space of weak solutions was studied for the complex main Vekua equation, while in \cite{camposbergman}, the space of classical $L_2$-solutions of the complex Vekua equation was analyzed, demonstrating the existence of a reproducing kernel. In both studies, the spaces were treated as {\it real} Hilbert spaces. Regarding the bicomplex main Vekua equation, in \cite{minevekua1}, the case when $f\in C^1(\overline{\Omega})\cap C^{2}(\Omega)\cap W^{2,\infty}(\Omega)$ was studied and it was shown that the Bergman space of classical $L_2$-solutions is a complex Hilbert space with a reproducing kernel \cite[Prop. 14]{minevekua1}.

The aim of this work is to develop the main properties of the Bergman spaces of $L_p$-weak solutions of Eq. \eqref{eqVekuaintro}. Following \cite{vekua0}, we introduce a notion of weak $\boldsymbol{\partial}$ and $\overline{\boldsymbol{\partial}}$ derivatives. The Vekua-Bergman space $\mathcal{A}_{(a,b)}^p(\Omega;\mathbb{B})$, where $\mathbb{B}$ denotes the algebra of bicomplex numbers, is defined as the collection of weak $L_p$-solutions of \eqref{eqVekuaintro} for $1<p\leqslant\infty$. The main properties such as completeness, reflexivity and separability are established. We analyze the regularity of the solutions of the non-homogeneous equation  $\overline{\boldsymbol{\partial}}W-aW-b\overline{W}=V$, with $V\in L_q(\Omega;\mathbb{B})$, $q>2$, and obtain that the solutions belong to the class $C^{0,1-\frac{2}{q}}(D;\mathbb{B})$ in every disk $D\Subset \Omega$. Consequently, the solutions of \eqref{eqVekuaintro} are continuous. We demonstrate that for every $z\in \Omega$,  the evaluation functionals $W\mapsto \operatorname{Sc}W(Z)$ and $W\mapsto \operatorname{Vec}W(z)$ are bounded with respect to the $L_p$-norm, where $\operatorname{Sc}W$ and $\operatorname{Vec}W$ are the scalar and vector part of $W$, respectively. For the case $p=2$, a kernel with a reproducing property is obtained. Unlike the complex case \cite{camposbergman}, where the similarity principle \cite{bers,vekua0} plays a crucial role, in the bicomplex case this property is not available for a general Vekua equation. Therefore, the proof relies on the regularity of the solutions and the use of an inequality of Calderon-Zygmun type. We also present a formula for the orthogonal projection onto $\mathcal{A}_{(a,b)}^2(\Omega;\mathbb{B})$, along with an explicit expression for the orthogonal complement $\left(\mathcal{A}_{(a,b)}^2(\Omega;\mathbb{B})\right)^{\perp}$. Finally, we extend the result that the scalar and vector parts of solutions of the main Vekua equation are weak solutions of the conductivity equations $\operatorname{div}f^2\nabla \left(\frac{u}{f}\right)$ and $\operatorname{div}f^{-2}\nabla (fv)=0$, respectively, where $f\in W^{1,\infty}(\Omega)$ is a non-vanishing complex-valued function such that $f^{-1}\in L_{\infty}(\Omega)$. This result is well-known in the case of classical solutions for $f\in C^2(\Omega)$ \cite{CamposBicomplex,pseudoanalyticvlad}. We show that for $f\in C^1(\Omega)$, the weak solutions of the main Vekua equation are classical solutions. For a star-shaped domain, we establish that for every solution $u$ of $\operatorname{div}f^2\nabla \left(\frac{u}{f}\right)$, a metaharmonic conjugate $v$ (such that $W=u+\mathbf{j}v$ is a solution of the main Vekua equation) can be obtained using a bounded operator in $W^{1,2}(\Omega)$. With the aid of such a relation, we extend the Runge property of elliptic equations \cite{Colton1,minerunge} to the main Vekua equation.

The paper is structured as follows. Section 2 provides an overview of the main results concerning the algebra of bicomplex numbers, analytic functions, and integral operators. In Section 3, the Vekua-Bergman space $\mathcal{A}_{(a,b)}^p(\Omega)$ is defined together with its main properties (completeness, separability, reflexivity). In section 4, the regularity of the solutions of the non-homogeneous Vekua equation and the boundedness of the evaluation functional are studied. Section 5 presents the existence and main properties of a reproducing kernel for $\mathcal{A}_{(a,b)}^2(\Omega;\mathbb{B})$. In Section 6, we derive an explicit expression for the orthogonal complement of $\mathcal{A}_{(a,b)}^2(\Omega;\mathbb{B})$ in $L_2(\Omega;\mathbb{B})$. In Section 7, the main properties of the solutions of the main Vekua equation are studied, as well as the relations of the scalar and vector parts of the solutions with the conductivity equation, the construction of metaharmonic conjugates, the existence of a Hilbert transform, and the Runge property.
\section{Background on bicomplex analytic functions}

This section provides a summary of the main properties of bicomplex numbers and bicomplex-valued analytic functions of a single complex variable, which can be found, for example, in \cite{CamposBicomplex,rochon,pseudoanalyticvlad,ShapiroBicomplex}.  We use the notation $\mathbb{N}_0:=\mathbb{N}\cup \{0\}$. Given Banach spaces $X$ and $Y$, $\mathcal{B}(X,Y)$ denotes the Banach space of bounded linear operators. When $X=Y$, we denote $\mathcal{B}(X)=\mathcal{B}(X,X)$. Additionally, $\mathcal{G}(X)$ denotes the group of bounded invertible operators and  $\mathcal{K}(X)$ denotes the ideal of compact operators.
\subsection{Basic properties}
Consider the set $\mathbb{C}^2$ with the usual addition and multiplication by an scalar. The bicomplex product $WZ$ of two elements $W=(w_1,w_2), Z=(z_1,z_2)\in \mathbb{C}^2$ is defined as $WZ=(w_1z_1-w_2z_2,w_1z_2+w_2z_1)$. With this product, $\mathbb{C}^2$ is a commutative algebra over $\mathbb{C}$ with unit element $(1,0)$ and is called the algebra of the {\it bicomplex  numbers}, denoted by $\mathbb{B}$. Denote $\mathbf{j}=(0,1)$. Identifying $a\in \mathbb{C}$ with $(a,0)$, we can write any $W\in \mathbb{B}$  in a unique form as $W=u+\mathbf{j}v$, with $u,v\in \mathbb{C}$. The complex numbers $\operatorname{Sc}W=u$ and $\operatorname{Vec}W=v$ are called the  {\it scalar} and {\it vector} parts of $W$, respectively. We have that $\mathbf{j}^2=-1$ and $i\mathbf{j}=\mathbf{j}i$.  The {\it bicomplex conjugate} of $W$ is defined by $\overline{W}:= u-\mathbf{j}v$. Note that $W\overline{W}=u^2-v^2\in \mathbb{C}$. The algebra $\mathbb{B}$ contains zero divisors. Indeed, consider 
\[
\mathbf{k}:=i\mathbf{j}, \qquad \mathbf{p}^{\pm} := \frac{1}{2}(1\pm\mathbf{k}),
\]
hence $\mathbf{p}^{+}\mathbf{p}^{-}=0$, $(\mathbf{p}^{\pm})^2=\mathbf{p}^{\pm}$ and $\mathbf{p}^{+}+\mathbf{p}^{-}=1$. The set of zero divisors of $\mathbb{B}$ is denoted by $\sigma(\mathbb{B})$. The following proposition summarizes some properties of the bicomplex numbers shown in \cite{CamposBicomplex, ShapiroBicomplex}.

\begin{proposition}\label{properties1}
	Let $W,V\in \mathbb{B}$.
	\begin{enumerate}
		\item $W\in \sigma(\mathbb{B})$ iff $W\overline{W}=0$.
		\item If $W\overline{W}\neq 0$, then $W$ is invertible and $W^{-1}= \frac{\overline{W}}{W\overline{W}}$.
		\item There exist unique $W^{\pm}\in \mathbb{C}$ such that $W=\mathbf{p}^{+}W^{+}+\mathbf{p}^{-}W^{-}$. Furthermore, $W^{\pm}=\operatorname{Sc}W\mp i\operatorname{Vec}W$.
		\item $W\in\sigma(\mathbb{B})$ iff $W=\mathbf{p}^{+}W^{+}$ or $W=\mathbf{p}^{-}W^{-}$.
		\item Denote $\mathcal{R}(\mathbb{B}):= \mathbb{B}\setminus \sigma(\mathbb{B})$. Then $\mathcal{R}(\mathbb{B})$ is the group of all the invertibles elements of $\mathbb{B}$, and for any $W\in \mathcal{R}(\mathbb{B})$ we have
		\begin{equation}\label{inversebicomplex}
			W^{-1}=\frac{\mathbf{p}^{+}}{W^{+}}+\frac{\mathbf{p}^{-}}{W^{-}}.
		\end{equation}
		\item The product $WV$ can be written as
		\begin{equation}\label{producbicomplex}
			WV= \mathbf{p}^{+}W^{+}V^{+}+\mathbf{p}^{+}W^{-}V^{-}.
		\end{equation}
		In particular
		\begin{equation}\label{powerbicomplex}
			W^n= \mathbf{p}^{+}(W^{+})^n+\mathbf{p}^{-}(W^{-})^n, \quad \forall n\in \mathbb{N}.
		\end{equation}
		When $W\in \mathcal{R}(\mathbb{B})$, formula (\ref{powerbicomplex}) is valid for all $n\in \mathbb{Z}$.
	\end{enumerate}	
\end{proposition}
For $W\in \mathbb{B}$ define 
\begin{equation}\label{involution1}
	W^{\dagger}:= \mathbf{p}^{+}(W^{+})^{*}+\mathbf{p}^{-}(W^{-})^{*}=\left(\operatorname{Sc}W\right)^*-\mathbf{j}\left(\operatorname{Vec}W\right)^*,	
\end{equation}
where $^{*}$ denotes the standard complex conjugation. The operation $W\mapsto W^{\dagger}$ is an involution in $\mathbb{B}$ \cite[Ch. I]{ShapiroBicomplex}. For $W,V\in \mathbb{B}$ define
\[
\langle W,V \rangle_{\mathbb{B}} := \operatorname{Sc} WV^{\dagger} = \frac{W^{+}(V^{+})^*+ W^{-}(V^{-})^*}{2}=(\operatorname{Sc}W)(\operatorname{Sc}V)^*+(\operatorname{Vec}W)(\operatorname{Vec}V)^*.
\]
It is not difficult to see that $\langle W,V \rangle_{\mathbb{B}}$ is a complex inner product in $\mathbb{B}$, and we have the norm
\[
|W|_{\mathbb{B}}:= \sqrt{\langle W,W \rangle_{\mathbb{B}}} = \sqrt{\frac{|W^{+}|^2+|W^{-}|^2}{2}}=\sqrt{|\operatorname{Sc}W|^2+|\operatorname{Vec}W|^2},
\]
where $|W^{\pm}|$ denotes the complex absolute value. The following inequality is immediate
\begin{equation}\label{equivalentnorms}
	\frac{1}{\sqrt{2}}|W^{\pm}|\leqslant |W|_{\mathbb{B}}\leqslant \frac{1}{\sqrt{2}}\left(|W^+|+|W^{-}|\right).
\end{equation}
\begin{remark}\label{remarknorm} 
	A direct computation shows the following relations:
	\begin{eqnarray}
		WV^{\dagger} & = & \langle W,V \rangle_{\mathbb{B}}+\mathbf{j}\langle W,\mathbf{j}V\rangle_{\mathbb{B}}, \label{bicomplexinnerproduct} \\
		\langle W,\mathbf{j}V\rangle_{\mathbb{B}} & =&-\langle \mathbf{j}W,V\rangle_{\mathbb{B}}, \label{bicomplexinnerproductwithj}\\
		|WV|_{\mathbb{B}} & \leqslant & \sqrt{2}|W|_{\mathbb{B}}|V|_{\mathbb{B}}.	\label{norminequality}
	\end{eqnarray}
	Inequality \eqref{norminequality} implies that the bicomplex product is a continuous operation in $\mathbb{B}$ and that $\mathcal{R}(\mathbb{B})$ is open \cite[Prop. 3]{CamposBicomplex}.
\end{remark}	

\subsection{Bicomplex weak Cauchy-Riemann operators and analytic functions}
Let $\Omega\subset \mathbb{C}$ be a bounded domain. Consider a bicomplex-valued function of the complex variable $z=x+iy\in \Omega$, $F:\Omega \rightarrow \mathbb{B}$. Hence $F=f_1+\mathbf{j}f_2$, where $f_1(z)=\operatorname{Sc}F(z)$ and $f_2(z)=\operatorname{Vec}F(z)$. By Proposition \ref{properties1}(3) we can write $F=\mathbf{p}^{+}F^{+}+\mathbf{p}^{-1}F^{-}$ where $F^{\pm}=f_1\mp if_2$. 

Let $\mathscr{F}(\Omega)$ be a linear space of complex-valued functions (e.g., $C^k(\Omega), L_p(\Omega), W^{k,p}(\Omega)$). We say that a bicomplex-valued function $F$ belongs to the bicomplex space $\mathscr{F}(\Omega;\mathbb{B})$ iff $\operatorname{Sc}F, \operatorname{Vec}F\in \mathscr{F}(\Omega)$, or equivalently, if $F^{\pm}\in \mathscr{F}(\Omega)$.

For $1\leqslant p<\infty$, $L_p(\Omega;\mathbb{B})$ is equipped with the norm $\|F\|_{L_p(\Omega; \mathbb{B})}:=\left(\iint_{\Omega}|F(z)|^p_{\mathbb{B}}dA_z\right)^{\frac{1}{p}}$, and for $p=\infty$, $\|F\|_{L_{\infty}(\Omega;\mathbb{B})}:= \operatorname{ess sup}_{z\in \Omega}|W(z)|_{\mathbb{B}}$.  In particular, for the case $p=2$, $L_2(\Omega; \mathbb{B})$ is a complex Hilbert space with the inner product 
\begin{equation}\label{L2bicomplexnorm}
	\langle F,G \rangle_{L_2(\Omega; \mathbb{B})} := \iint_{\Omega}\langle F(z),G(z)\rangle_{\mathbb{B}}dA_z, 
\end{equation}
Since $L_p(\Omega; \mathbb{B})$ can be regarded as the product space $L_p(\Omega)\times L_p(\Omega)$ and for $1\leqslant p<\infty$$, L_p(\Omega)$ is separable \cite[Sec. 4.3]{brezis},  $L_p(\Omega;\mathbb{B})$ is also separable. In particular, $L_2(\Omega;\mathbb{B})$ is a separable complex Hilbert space. By the same argument, since $L_p(\Omega)$ is reflexive for $1<p<\infty$ \cite[Th. 4.10]{brezis}, $L_p(\Omega;\mathbb{B})$ is also reflexive.  The Sobolev spaces $W^{k,p}(\Omega; \mathbb{B})$, with $k\in \mathbb{N}$ and $1\leqslant p\leqslant \infty$, the local spaces $L_{p,loc}(\Omega; \mathbb{B})$, $W^{k,p}_{loc}(\Omega;\mathbb{B})$, and $W_0^{k,p}(\Omega; \mathbb{B})$ are defined in the usual way and equipped with corresponding norms.

Let $X\subset \mathbb{C}$ and $\epsilon\in (0,1]$. We recall that a function $W$ satisfies the H\"older condition with H\"older exponent $\epsilon$ on $X$ if $ \sup\limits_{\overset{z_1,z_2\in X}{ z_1\neq z_2}}\frac{|W(z_2)-W(z_1)|_{\mathbb{B}}}{|z_2-z_1|^{\epsilon}}<\infty$, and we denote $W\in C^{0,\epsilon}(X;\mathbb{B})$. In the case $\epsilon=1$, we say that $W$ is Lipschitz continuous on $X$. For a bounded domain $\Omega$, a function $W:\Omega\rightarrow \mathbb{B}$ is said to be {\it locally H\"older on disks} of $\Omega$ with H\"older exponent $\epsilon$ if $W|_D\in C^{0,\epsilon}(\overline{D};\mathbb{B})$ for every disk $D\Subset \Omega$. In such case, we denote $W\in C_{locdisk}^{0,\epsilon}(\Omega)$.

\begin{remark}\label{remarksobolevembedding}
	By the standard Sobolev embedding theorems (see \cite[Cor. 9.14]{brezis} and \cite{evans}, Th. 5. from Sec. 5.6 and Th. 4 from Sec. 5.8) , when $\Omega$ is of class $C^1$, we have the compact embeddings: $W^{1,p}(\Omega)\hookrightarrow L_q(\Omega)$ for $p<2$ and $2<q<p^*:=\frac{2p}{p-1}$; $W^{1,2}(\Omega)\hookrightarrow L_q(\Omega)$ for  $2\leqslant q<\infty$; and $W^{1,p}(\Omega)\subset C^{0,1-\frac{2}{p}}(\overline{\Omega})$ for $2<p\leqslant \infty$. In consequence, the following embeddings are compact:  $W^{1,p}(\Omega;\mathbb{B})\hookrightarrow L_q(\Omega;\mathbb{B})$ for $p<2$ and $2<q<p^*$; $W^{1,2}(\Omega; \mathbb{B})\hookrightarrow L_q(\Omega; \mathbb{B})$ for $2\leqslant q<\infty$; and $W^{1,p}(\Omega; \mathbb{B})\hookrightarrow C^{0,1-\frac{2}{p}}(\overline{\Omega}; \mathbb{B})$ for $2<p\leqslant \infty$. Consequently, the local Sobolev spaces are embedded as follows: $W^{1,p}_{loc}(\Omega;\mathbb{B})\hookrightarrow L_q(\Omega;\mathbb{B})$ for $p<2$ and $2<q<p^*$; $W^{1,2}_{loc}(\Omega)\hookrightarrow L_{q,loc}(\Omega)$ for $2\leqslant q<\infty$; $W^{1,p}_{loc}(\Omega)\hookrightarrow C_{locdisk´}^{0,1-\frac{2}{p}}(\Omega)$ for $2<p\leqslant \infty$.
\end{remark}

Given $z=x+iy$ with $x,y\in \mathbb{R}$, denote
\begin{equation}\label{bicomplexification}
	\widehat{z}:=x+\mathbf{j}y \in \mathbb{B}.	
\end{equation}
Note that when $z,z_0\in \mathbb{C}$, 
\begin{equation}\label{variablez}
	\widehat{z}-\widehat{z_0}= \mathbf{p}^{+}(z^*-z^*_0)+\mathbf{p}^{-}(z-z_0).
\end{equation}
From Proposition \ref{properties1}(5), when $z\neq z_0$ we have
\begin{equation}\label{powerz}
	\left(\widehat{z}-\widehat{z_0}\right)^n = \mathbf{p}^{+}(z^*-z_0^*)^n+\mathbf{p}^{-}(z-z_0)^n, \quad \forall n\in \mathbb{Z}.
\end{equation}
A direct computation shows that $\mathbb{C}\ni z\mapsto \widehat{z}\in \mathbb{B}$ is a monomorphism of algebras. Consequently, the set $\{x+\mathbf{j}y\,|\, x,y\in \mathbb{R}\}$ is a subfield of $\mathbb{B}$ isomorphic to $\mathbb{C}$,  and $\widehat{z}^{\dagger}= \overline{\widehat{z}}= \widehat{z^*}$.

The Bicomplex {\it Cauchy-Riemann} operators are defined by
\begin{equation}\label{defCR}
	\boldsymbol{\partial} := \frac{1}{2}\left(\frac{\partial}{\partial x}-\mathbf{j}\frac{\partial}{\partial y}\right), \quad \overline{\boldsymbol{\partial}} :=\frac{1}{2}\left(\frac{\partial}{\partial x}+\mathbf{j}\frac{\partial}{\partial y}\right).
\end{equation} The following relations hold:
\begin{equation}\label{CRdecomposition}
	\overline{\boldsymbol{\partial}}= \mathbf{p}^{+}\frac{\partial}{\partial z}+\mathbf{p}^{-}\frac{\partial}{\partial z^*}  , \quad \boldsymbol{\partial} = \mathbf{p}^{+}\frac{\partial}{\partial z^*}+\mathbf{p}^{-}\frac{\partial}{\partial z},
\end{equation}
where $\frac{\partial}{\partial z}:= \frac{1}{2}\left(\frac{\partial}{\partial x}-i\frac{\partial}{\partial y}\right)$ and $\frac{\partial}{\partial z^*}:= \frac{1}{2}\left(\frac{\partial}{\partial x}+i\frac{\partial}{\partial y}\right)$ are the usual complex Cauchy-Riemann operators (see \cite{CamposBicomplex}). A function $W\in C^1(\Omega; \mathbb{B})$ is called $\mathbb{B}$-analytic (anti-analytic) when $\overline{\boldsymbol{\partial}}W=0$ ($\boldsymbol{\partial} W=0$). From Proposition \ref{properties1}(5) and (\ref{CRdecomposition}), a function  $W\in C^1(\Omega;\mathbb{B})$ is $\mathbb{B}$-analytic (anti-analytic) iff $W^{\pm}$ are anti-analytic and analytic, respectively, in the complex sense. By (\ref{CRdecomposition}) and (\ref{powerz}), the powers $\{(\widehat{z}-\widehat{z_0})^n\}_{n=0}^{\infty}$ are $\mathbb{B}$-analytic and $\boldsymbol{\partial}(\widehat{z}-\widehat{z_0})^n=n(\widehat{z}-\widehat{z_0})^{n-1}$. The standard properties of complex analytic functions (Cauchy integral and power series representations, maximum principle, Liouville theorem, etc) are also valid for $\mathbb{B}$-analytic functions (see, e.g., \cite[Remark 3.2]{camposmendez} or \cite[Prop. 6]{minevekua1}).

Let $1\leqslant p\leqslant \infty$ and $W\in L_p(\Omega; \mathbb{B})$. As is usual, a function $V\in L_{1,loc}(\Omega; \mathbb{B})$ is called the weak $\boldsymbol{\partial}$-derivative of $W$ is the following condition is fulfilled
\begin{equation}\label{defweakcomplexpartialderivative}
	\iint_{\Omega}W(z)\boldsymbol{\partial}\phi(z)dA_z=-\iint_{\Omega}V(z)\phi(z)dA_z, \quad \forall \phi\in C_0^{\infty}(\Omega; \mathbb{B}),
\end{equation}
and we denote $V=\boldsymbol{\partial}W$. The weak $\overline{\boldsymbol{\partial}}$-derivative is defined analogously.
\begin{remark}\label{remarkWeyllema}
	From relations (\ref{CRdecomposition}), taking scalar functions in (\ref{defweakcomplexpartialderivative}) we have that $\overline{\boldsymbol{\partial}}W=0$ weakly iff $\frac{\partial W^+}{\partial z}=\frac{\partial W^{-}}{\partial z^*}=0$ weakly. By the Weyl Lemma \cite[Ch. 18, Cor. 4.11]{conway2}, $W^+$ and $W^{-}$ are analytic and anti-analytic and hence $W$ is $\mathbb{B}$-analytic. In the same way, $\boldsymbol{\partial}W=0$ weakly iff $W$ is $\mathbb{B}$-anti-analytic.
\end{remark}
 Following the notations from \cite[Sec. 5]{vekua0}, we introduce the  spaces
\[
\mathfrak{D}_p(\Omega; \mathbb{B}):=\left\{W\in L_p(\Omega; \mathbb{B})\, |\, \boldsymbol{\partial}W\in L_p(\Omega; \mathbb{B}) \right\}, \;  \overline{\mathfrak{D}}_p(\Omega; \mathbb{B}):=\left\{W\in L_p(\Omega; \mathbb{B})\, |\, \overline{\boldsymbol{\partial}}W\in L_p(\Omega; \mathbb{B}) \right\}.
\]
Note that $W\in W^{1,p}(\Omega; \mathbb{B})$ iff $W\in \mathfrak{D}_p(\Omega; \mathbb{B})\cap \overline{\mathfrak{D}}_p(\Omega; \mathbb{B})$, and $W^{1,p}(\Omega; \mathbb{B})$ can be endowed with the norm $\|W\|_{W^{1,p}(\Omega; \mathbb{B})}=\left(\|W\|_{L_p(\Omega; \mathbb{B})}^p+\|\boldsymbol{\partial}W\|_{L_p(\Omega; \mathbb{B})}^p+\|\overline{\boldsymbol{\partial}}W\|_{L_p(\Omega; \mathbb{B})}^p\right)^{\frac{1}{p}}$ for $1\leqslant p<\infty$, and $\|W\|_{W^{1,\infty}(\Omega; \mathbb{B})}=\max\{\|W\|_{L_{\infty}(\Omega; \mathbb{B})},\|\boldsymbol{\partial}W\|_{L_{\infty}(\Omega; \mathbb{B})},\|\overline{\boldsymbol{\partial}}W\|_{L_{\infty}(\Omega; \mathbb{B})} \}$.

\begin{remark}\label{completenessweakderivative}
	\begin{itemize}
		\item[(i)] The spaces $\mathfrak{D}_p(\Omega; \mathbb{B})$ and $\overline{\mathfrak{D}}_p(\Omega; \mathbb{B})$ are complete with their respective norms. The proof is similar to that for the completeness of Sobolev spaces.
		\item[(ii)] A direct computation shows that
		\begin{equation}\label{involutuionofcauchyriemannoperator}
			(\boldsymbol{\partial}W)^{\dagger}= \overline{\boldsymbol{\partial}}W^{\dagger} \quad \mbox{ and  }\quad (\overline{\boldsymbol{\partial}}W)^{\dagger}= \boldsymbol{\partial} W^{\dagger}.
		\end{equation} 
		 Thus, we have the bijection $\mathfrak{D}_p(\Omega; \mathbb{B}) \overset{\dagger}{\leftrightarrow}\overline{\mathfrak{D}}_p(\Omega; \mathbb{B})$. Additionally, $\mathfrak{D}_p(\Omega; \mathbb{B})$ and $\overline{\mathfrak{D}}_p(\Omega; \mathbb{B})$ are $\mathbb{B}$-modules and $\boldsymbol{\partial}$ and $\overline{\boldsymbol{\partial}}$ are $\mathbb{B}$-linear maps.
	\end{itemize}
\end{remark}

\subsection{Bicomplex integral operators}
From now on, we suppose that $\Omega$ is a bounded simply connected domain. 
Given $g\in L_p(\Omega)$ ($1<p\leqslant \infty)$, the complex Theodorescu operator (and its conjugate) are given by
\[
A_{\Omega}g(z):= \frac{1}{\pi}\iint_{\Omega}\frac{g(\zeta)}{z-\zeta}dA_{\zeta}, \quad B_{\Omega}g(z):= \frac{1}{\pi}\iint_{\Omega}\frac{g(\zeta)}{z^*-\zeta^*}dA_{\zeta}.
\]
Note that $B_{\Omega}g(z)=(A_{\Omega}g^*(z))^*$. 
\begin{remark}\label{proptheodorescucomplex}
	The following statements hold.
	\begin{itemize}
		\item[(i)] For $1\leqslant p\leqslant\infty$, $A_{\Omega}, B_{\Omega}\in \mathcal{B}(L_p(\Omega))$ and their norms are bounded by $2\operatorname{diam}(\Omega)$ (see \cite[Sec. 12.13]{karapetyantskrav} and \cite[Th. 6.18]{folland}). Furthermore, $\frac{\partial}{\partial z^*}A_{\Omega}g=g$ and $\frac{\partial}{\partial z}B_{\Omega}g=g$ (in the weak sense) for all $g\in L_p(\Omega)$ \cite[Th. 1.14]{vekua0}.
		\item[(ii)] For $1< p<\infty$, $A_{\Omega}, B_{\Omega}\in \mathcal{B}\left(L_p(\Omega),W^{1,p}(\Omega)\right)$ (see  \cite[Appendix A]{leblont} and the references cited there). For $p=\infty$, $A_{\Omega},B_{\Omega}\in \mathcal{B}(L_{\infty}(\Omega), C^{0,\epsilon}(\Omega))$ for all $0<\epsilon<1$ (see \cite[Th 6.1]{Colton1}). 
		\item[(iii)] The adjoint of $A_{\Omega}$ in $L_2(\Omega)$ is $-B_{\Omega}$. The proof is a straightforward.
	\end{itemize}
\end{remark}

Now, assume that $\Omega$ is a Lipschitz domain, in order that $\Gamma=\partial \Omega$ forms a Jordan curve with a Lipschitz exterior normal $\nu(\zeta)$, and such that Green's theorem  and the existence of the Sobolev spaces $W^{1,1-\frac{1}{p}}(\Gamma)$ hold \cite[Ch. 3]{maclean}. For $\varphi\in L_p(\Gamma;\mathbb{B})$, we define
\[
\int_{\Gamma}\varphi(\zeta)d\widehat{\zeta}:= \mathbf{p}^+\int_{\Gamma}\varphi^+(\zeta)d\zeta^*+\mathbf{p}^{-}\int_{\Gamma}\varphi^{-}(\zeta)d\zeta\; \mbox{ and }\; \int_{\Gamma}\varphi(\zeta)d\widehat{\zeta}^{\dagger}= \left(\int_{\Gamma}\varphi^{\dagger}(\zeta)d\zeta\right)^{\dagger}.
\]
A direct computation using Green's theorem shows the following relations for $W\in W^{1,p}(\Omega; \mathbb{B})$:
\begin{align}\label{bicomplexGreenidentities}
		\iint_{\Omega}\overline{\boldsymbol{\partial}}W(z)dA_z =\frac{1}{2\pi \mathbf{j}}\int_{\Gamma}\operatorname{tr}_{\Gamma}W(\zeta)d\widehat{\zeta}\;\; \mbox{ and }\; \; \iint_{\Omega}\boldsymbol{\partial}W(z)dA_z =-\frac{1}{2\pi \mathbf{j}}\int_{\Gamma}\operatorname{tr}_{\Gamma}W(\zeta)d\widehat{\zeta}^{\dagger}.
\end{align}
We recall the Borel-Pompeiu formula, which is valid for $\varphi\in C^{\infty}(\overline{\Omega})$:
\begin{equation}\label{BorelPompeiuformulacomplex}
	C_{\Gamma}\varphi(z)+A_{\Omega}\left[\frac{\partial \varphi}{\partial z^*}\right](z)= \varphi(z), \quad z\in \Omega.
\end{equation}
Here, $\displaystyle C_{\Gamma}\varphi(z):=\frac{1}{2\pi i}\int_{\Gamma}\frac{\varphi(\zeta)}{\zeta-z}d\zeta$ is the (regular) Cauchy integral operator. Actually, $C_{\Gamma}\varphi(z)$ is well defined for $\varphi\in L_p(\Gamma)$, $1\leqslant p\leqslant \infty$, and it is well known that $C_{\Gamma}\varphi$ is analytic in $\mathbb{C}\setminus \Gamma$.

\begin{proposition}\label{propkcauchyintegraloperator}
	For $1<p<\infty$, the Cauchy integral operator $C_{\Gamma}: W^{1,1-\frac{1}{p}}(\Gamma)\rightarrow W^{1,p}(\Omega)$ is bounded.
\end{proposition}
The proof is provided in Appendix A. In particular, the Borel-Pompeiu formula (\ref{BorelPompeiuformulacomplex}) is valid for functions $u\in W^{1,p}(\Omega)$. Following \cite{CamposBicomplex}, we introduce the bicomplex version of the Theodorescu and the Cauchy integral operators.

\begin{definition}\label{definitionBicomplextheodorescu}
	Let $1<p\leqslant \infty$. The $\mathbb{B}$-Theodorescu and the $\mathbb{B}$-Cauchy operators are defined by
	\begin{align}\label{BicomplextheodorescuandCauchy}
		\mathbf{T}_{\Omega}W(z) :=\frac{1}{\pi} \iint_{\Omega}\frac{W(\zeta)}{\widehat{z}-\widehat{\zeta}}dA_{\zeta},  \; \mbox{ and }\; 
		\mathbf{C}_{\Gamma}\varphi(z)  := \frac{1}{2\pi\mathbf{j}}\int_{\Gamma}\frac{\varphi(\zeta)}{\widehat{\zeta}-\widehat{z}}d\widehat{\zeta}, 
	\end{align}
	for $W\in L_p(\Omega;\mathbb{B})$ and $\psi\in L_p(\Gamma,\mathbb{B})$, respectively.
\end{definition}
\begin{proposition}\label{remarktheodorescuop}
	The following statements hold.
	\begin{itemize}
		\item[(i)] The operators in \eqref{BicomplextheodorescuandCauchy} can be written as $	\mathbf{T}_{\Omega}W(z)= \mathbf{p}^+B_{\Omega}W^+(z)+\mathbf{p}^{-}A_{\Omega}W^{-}(z)$ and $\mathbf{C}_{\Gamma}\varphi(z) = \mathbf{p}^+\left(C_{\Gamma}\left(\varphi^+\right)^*(z)\right)^*+\mathbf{p}^{-}C_{\Gamma}\varphi^{-}(z)$, respectively.
		\item[(ii)] For $1\leqslant p\leqslant \infty$, $\mathbf{T}_{\Omega}\in \mathcal{B}\left(L_p(\Omega; \mathbb{B})\right)$ and $\mathbf{T}_{\Omega}W\in \overline{\mathfrak{D}}_p(\Omega;\mathbb{B})$ with  $\overline{\boldsymbol{\partial}}\mathbf{T}_{\Omega}W=W$ for every $W\in L_p(\Omega;\mathbb{B})$. 
		\item[(iii)] For $1< p<\infty$, $\mathbf{T}_{\Omega}\in \mathcal{B}\left(L_p(\Omega; \mathbb{B}), W^{1,p}(\Omega; \mathbb{B})\right)$, and for $p=\infty$, $\mathbf{T}_{\Omega}\in \mathcal{B}(L_{\infty}(\Omega),C^{0,\epsilon}(\Omega;\mathbb{B}))$ for all $0<\epsilon<1$.
		\item[(iv)]  If $\Omega$ is of class $C^1$, then $\mathbf{T}_{\Omega}\in \mathcal{K}\left(L_p(\Omega;\mathbb{B})\right)$ for all $1< p<\infty $.
		\item[(v)] $\mathbf{C}_{\Gamma}\in \mathcal{B}\left(W^{1,1-\frac{1}{p}}(\Gamma,\mathbb{B}), W^{1,p}(\Omega;\mathbb{B})\right)$ for $1< p<\infty$.  
		\item[(vi)] Given $W\in W^{1,p}(\overline{\Omega};\mathbb{B})$, $1<p<\infty$, the bicomplex Borel-Pompeiu formula is valid:
		\begin{equation}\label{bicomplexBorelpompeiu}
			\mathbf{C}_{\Gamma}[\operatorname{tr}_{\Gamma}W](z)+\mathbf{T}_{\Omega}\left[\overline{\boldsymbol{\partial}}W\right](z)= W(z), \quad z\in \Omega.
		\end{equation}
		In particular, $\mathbf{T}\overline{\boldsymbol{\partial}}V =V$ for all $V\in W_0^{1,p}(\Omega; \mathbb{B})$.  
		\item[(vii)] The adjoint of $\mathbf{T}_{\Omega}$ in the space $L_2(\Omega; \mathbb{B})$ is given by
		\begin{equation}\label{adjointtheodorescubicomplex}
			\mathbf{T}_{\Omega}^{*}W=-\mathbf{p}^+A_{\Omega}W^{+}-\mathbf{p}^{-}B_{\Omega}W^{-}.
		\end{equation}
		Additionally, $-\boldsymbol{\partial}\mathbf{T}^{*}_{\Omega}W=W$ for all $W\in L_2(\Omega; \mathbb{B})$, and  $-\mathbf{T}_{\Omega}^{*}\boldsymbol{\partial}W=W$ for $W\in W_0^{1,2}(\Omega)$. 
	\end{itemize}
\end{proposition}
\begin{proof}
	\begin{itemize}
		\item[(i)] It follows from the definitions of $A_{\Omega}$, $B_{\Omega}$, $C_{\Gamma}$ and \eqref{powerz}.
		\item[(ii)] It is a consequence of the decomposition in the previous point and Remark \ref{proptheodorescucomplex}(i)-(ii).
		\item[(iii)] It follows from Remark \ref{proptheodorescucomplex} and point (i). 
		\item[(iv)] Indeed, by Remark \ref{remarksobolevembedding},  the embeddings $W^{1,2}(\Omega;\mathbb{B})\hookrightarrow L_2(\Omega;\mathbb{B})$ and $W^{1,p}(\Omega;\mathbb{B})\hookrightarrow C^{0,1-\frac{2}{p}}(\overline{\Omega};\mathbb{B})$ (for $p>2$) are compact.  Hence, the relations $L_2(\Omega;\mathbb{B})\overset{\mathbf{T}_{\Omega}}{\rightarrow}W^{1,2}(\Omega;\mathbb{B})\hookrightarrow L_2(\Omega;\mathbb{B})$ and $L_p(\Omega;\mathbb{B})\overset{\mathbf{T}_{\Omega}}{\rightarrow}W^{1,p}(\Omega;\mathbb{B})\hookrightarrow C^{0,1-\frac{2}{p}}(\overline{\Omega};\mathbb{B})\hookrightarrow L_p(\Omega;\mathbb{B})$ (for $p>2$) imply that $\mathbf{T}_{\Omega}\in \mathcal{K}\left(L_p(\Omega;\mathbb{B})\right)$. For $1<p<2$, note that $p^*=\frac{2p}{p-2}>2$, hence taking $q\in (2,p^*)$ we have the compact embedding $W^{1,p}(\Omega;\mathbb{B})\hookrightarrow L_q(\Omega;\mathbb{B})$ and the relations $L_p(\Omega;\mathbb{B})\overset{\mathbf{T}_{\Omega}}{\rightarrow}W^{1,p}(\Omega;\mathbb{B})\hookrightarrow L_q(\Omega;\mathbb{B})\rightarrow L_p(\Omega;\mathbb{B})$. Consequently, $\mathbf{T}_{\Omega}\in \mathcal{K}\left(L_p(\Omega;\mathbb{B})\right)$.
		\item[(v)] It is a consequence of Proposition \ref{propkcauchyintegraloperator} and point (i). 
		\item[(vi)] Follows from point (i) and the complex Borel-Pompeiu formula.
		\item[(vii)] This is a consequence of \ref{proptheodorescucomplex}(iii). 
	\end{itemize}
\end{proof}

We recall that $\Gamma$ is said to be a Liapunov surface if it is of class $C^1$ and the angle $\Theta(z)$, $z\in \Gamma$, between the tangent to the curve at the point $z$ and the positive real axis satisfies a H\"older condition \cite[Sec. 1.1.1]{gohberg}.
\begin{lemma}\label{lemalineintegral}
	Suppose that $\Gamma$ is a Liapunov curve. Let $1<p<\infty$, $p'=\frac{p}{p-1}$, and  $V\in W^{1,p}(\Omega; \mathbb{B})$ which satisfies the following condition:
	\begin{equation}\label{conditionlineintegral}
		\int_{\Gamma}\psi(\zeta)V(\zeta)d\widehat{\zeta}=0, \quad \forall \psi\in W^{1,1-\frac{1}{p'}}(\Gamma; \mathbb{B}).
	\end{equation}
	Then there exists $W_0\in W_0^{1,p}(\Omega; \mathbb{B})$ and $G\in W^{1,p}(\Omega; \mathbb{B})$ $\mathbb{B}$-analytic such that
	\begin{equation}\label{decompositionlineintegralcondition}
		V=W_0+G.
	\end{equation}
\end{lemma}
\begin{proof}
	The condition (\ref{conditionlineintegral}) can be rewritten as follows:
	\begin{equation}\label{auxilia1}
		\int_{\Gamma}\psi(\zeta)\operatorname{tr}_{\Gamma}u(\zeta)d\zeta^*= 0, \quad \int_{\Gamma}\phi(\zeta)\operatorname{tr}_{\Gamma}v(\zeta)d\zeta=0 \quad \forall \phi,\psi\in W^{1,1-\frac{1}{p'}}(\Gamma),
	\end{equation}
	where $u=V^+$ and $v=V^{-}$. We focus on the second integral. Condition \eqref{auxilia1} implies $\displaystyle \int_{\Gamma}\operatorname{tr}_{\Gamma}v(\zeta)\zeta^nd\zeta=0$ for all $n\in \mathbb{N}_0$. According to \cite[Th. 4.5 of Sec. 2.4]{gohberg},  $\operatorname{tr}_{\Gamma}u$ is the non-tangential limit of $g(z)=C_{\Gamma}[\operatorname{tr}_{\Gamma}v](z)$, $z\in \Omega$. By Proposition \ref{propkcauchyintegraloperator}, $g\in W^{1,p}(\Omega)$ and then $\operatorname{tr}_{\Gamma}u=\operatorname{tr}_{\Gamma}g$. Hence $w_0=v-g\in W^{1,p}_0(\Omega)$ and $v=w_0+g$. Applying the same procedure to the integral $\int_{\Gamma}\psi^*(\zeta)\operatorname{tr}_{\Gamma}u^*(\zeta)d\zeta=0$, we obtain that $u=w_1+h$, where $w_1\in W_0^{1,p}(\Omega)$ and $h\in W^{1,p}(\Omega)$ is anti-analytic. Thus, $V=W+G$, where $W=\mathbf{p}^+w_1+\mathbf{p}^{-}w_0\in W^{1,p}(\Omega; \mathbb{B})$ and $G=\mathbf{p}^+h+\mathbf{p}^{-}g\in W^{1,p´}(\Omega; \mathbb{B})$, which is $\mathbb{B}$-analytic.
\end{proof}

\section{The Vekua-Bergman space}

Let $1<p\leqslant \infty$ and $a,b\in L_{\infty}(\Omega;\mathbb{B})$. We consider the Bicomplex Vekua equation
\begin{equation}\label{bicomplexvekua}
	\overline{\boldsymbol{\partial}}W(z)= a(z)W(z)+b(z)\overline{W(z)}, \quad z\in \Omega.
\end{equation}
According to \eqref{defweakcomplexpartialderivative} and by \eqref{involutuionofcauchyriemannoperator}, a function $W\in L_p(\Omega; \mathbb{B})$ is a weak solution of the Vekua equation if and only if it satisfies the condition
\begin{equation}\label{weaksolutionsecondform}
	-\iint_{\Omega} W(z)\left(\boldsymbol{\partial}V(z)\right)^{\dagger}dA_z= \iint_{\Omega} \left(a(z)W(z)+b(z)\overline{W(z)}\right) V^{\dagger}(z)dA_z, \quad \forall V\in C_0^{\infty}(\Omega; \mathbb{B}).
\end{equation}
For $1<p<\infty$, it is a straightforward to show, using \eqref{norminequality}, that \eqref{weaksolutionsecondform} holds for all $V\in W^{1,p'}_0(\Omega;\mathbb{B})$, where $p'=\frac{p}{p-1}$.
\begin{definition}
	The {\bf Vekua-Bergman space} is the class of all weak solutions $W\in L_p(\Omega; \mathbb{B})$ of Eq. (\ref{bicomplexvekua}), denoted by $\mathcal{A}_{(a,b)}^p(\Omega; \mathbb{B})$.
\end{definition}

\begin{remark}\label{reamkWeyllemma}
	For the case $a\equiv b \equiv 0$, we obtain the $\mathbb{B}$-analytic Bergman space $\mathcal{A}^p(\Omega; \mathbb{B}):= \mathcal{A}_{(0,0)}^p(\Omega; \mathbb{B})$. By Remark \ref{remarkWeyllema}, $\mathcal{A}^p(\Omega; \mathbb{B})=\{ W\in C^1(\Omega; \mathbb{B})\cap L_p(\Omega; \mathbb{B})\, |\, W \mbox{ is }  \mathbb{B}\mbox{-analytic}\}$.
\end{remark}

Denote $\mathbf{Q}_{(a,b)}W:= aW+b\overline{W}$. Since $a,b\in L_{\infty}(\Omega; \mathbb{B})$, then $\mathbf{Q}_{(a,b)}\in \mathcal{B}\left(L_p(\Omega)\right)$. By (\ref{norminequality}), we have the inequality
\begin{equation}\label{normoperatorQ}
	\|\mathbf{Q}_{(a,b)}\|_{\mathcal{B}(L_p(\Omega;\mathbb{B}))}\leqslant \sqrt{2}\max\{\|a\|_{L_{\infty}(\Omega;\mathbb{B})}, \|b\|_{L_{\infty}(\Omega;\mathbb{B})}\}.
\end{equation}
 Following \cite{hugo,delgadoleblond}, we introduce the operator 
\begin{equation}\label{operatorS}
	\mathbf{S}_{\Omega}^{(a,b)} W:=W-\mathbf{T}_{\Omega}[\mathbf{Q}_{(a,b)}W]. 
\end{equation}
Again, $\mathbf{S}_{\Omega}^{(a,b)}\in \mathcal{B}(L_p(\Omega; \mathbb{B}))$. When $\Omega$ is of class $C^1$ and $1< p<\infty$, by Proposition \ref{remarktheodorescuop}(iv), $\mathbf{S}_{\Omega}^{(a,b)}$ is a Fredholm operator with index $0$. In particular, it possesses a finite-dimensional kernel \cite[Th. 2.22]{maclean}.

\begin{proposition}\label{propositiondiferenciability} The following statements hold.
\begin{itemize}
	\item[(i)] The operator $\mathbf{S}_{\Omega}^{(a,b)}$ maps $\overline{\mathfrak{D}}_p(\Omega;\mathbb{B})$ into itself, and the following relation is valid:
	\begin{equation}\label{relationpartialS}
		\overline{\boldsymbol{\partial}}\mathbf{S}_{\Omega}^{(a,b)}W=\left(\overline{\boldsymbol{\partial}}-\mathbf{Q}_{(a,b)}\right)W,\quad \forall W\in \overline{\mathfrak{D}}_p(\Omega;\mathbb{B}).
	\end{equation}
\item[(ii)] 	Let $1<p<\infty$. If $W\in L_p(\Omega; \mathbb{B})$ and $V\in L_p(\Omega; \mathbb{B})$ satisfy $\overline{\boldsymbol{\partial}}W-\mathbf{Q}_{(a,b)}W=V$ in the weak sense, then $W\in W^{1,p}_{loc}(\Omega; \mathbb{B})$. 
\item[(iii)]  If $W\in L_{\infty}(\Omega; \mathbb{B})$ and $V\in L_{\infty}(\Omega; \mathbb{B})$ satisfy $\overline{\boldsymbol{\partial}}W-\mathbf{Q}_{(a,b)}W=V$ in the weak sense, then $W\in C^{0,\epsilon}_{locdisk}(\Omega; \mathbb{B})$ for all $0<\epsilon <1$. 
\end{itemize}	
\end{proposition}

\begin{proof}
	Consider $W\in \overline{\mathfrak{D}}_p(\Omega; \mathbb{B})$, and let $H=\mathbf{S}_{\Omega}^{(a,b)}W\in L_p(\Omega; \mathbb{B})$. As $\mathbf{Q}_{(a,b)}W\in L_p(\Omega; \mathbb{B})$, by Proposition \ref{remarktheodorescuop}(ii), we observe that  $\mathbf{T}_{\Omega} \mathbf{Q}_{(a,b)}W\in \overline{\mathfrak{D}}_p(\Omega; \mathbb{B})$, hence $H\in \overline{\mathfrak{D}}_p(\Omega;\mathbb{B})$, and by Proposition \ref{remarktheodorescuop}(ii), $\overline{\boldsymbol{\partial}}H=\overline{\boldsymbol{\partial}}W-\mathbf{Q}_{(a,b)}W$. This establishes (i). Now suppose that $1<p<\infty$. Since $W\in L_p(\Omega;\mathbb{B})$ is a solution of the non-homogenous Vekua equation with right-hand side $V\in L_p(\Omega;\mathbb{B})$, $W\in \overline{\mathfrak{D}}_p(\Omega;\mathbb{B})$. By \eqref{relationpartialS}, $H=\mathbf{S}_{\Omega}^{(a,b)}W$ is a weak solution of $\overline{\boldsymbol{\partial}}H=V$ if and only if $W$ is a weak solution of $(\overline{\boldsymbol{\partial}}-\mathbf{Q}_{(a,b)})W=V$. Note that equation $\overline{\boldsymbol{\partial}}H=V$ and Remark \ref{reamkWeyllemma} imply that $H=\mathbf{T}_{\Omega}V+G$, with $G\in \mathcal{A}^p(\Omega; \mathbb{B})$. Hence $H\in W^{1,p}_{loc}(\Omega; \mathbb{B})$, by Proposition \ref{remarktheodorescuop}(iii). Thus, $W=H+\mathbf{T}_{\Omega}\mathbf{Q}_{(a,b)}W\in W^{1,p}_{loc}(\Omega)$. When $p=\infty$, by Proposition \ref{remarktheodorescuop}(iii), $\mathbf{T}_{\Omega}V$ and $\mathbf{T}_{\Omega}\mathbf{Q}_{(a,b)}W$ belongs to $C^{0,\epsilon}(\Omega;\mathbb{B})$ for all $0<\epsilon<1$, from where we obtain (iii).
\end{proof}

\begin{corollary}\label{coroOperatorS}
	$H=\mathbf{S}_{\Omega}^{A,B}W\in \mathcal{A}^p(\Omega; \mathbb{B})$ iff $W\in \mathcal{A}_{(a,b)}^p(\Omega; \mathbb{B})$. Hence $\mathcal{A}_{(a,b)}^p(\Omega; \mathbb{B})\subset W^{1,p}_{loc}(\Omega; \mathbb{B})$ for $1<p<\infty$, and $\mathcal{A}_{(a,b)}^{\infty}(\Omega;\mathbb{B})\subset C^{0,\epsilon}_{locdisk}(\Omega;\mathbb{B})$ for all $0<\epsilon<1$.
\end{corollary}

Consider the unbounded operator $\overline{\boldsymbol{\partial}}: \operatorname{dom}(\overline{\boldsymbol{\partial}})\subset L_p(\Omega;\mathbb{B})\rightarrow L_p(\Omega;\mathbb{B})$ with domain $\operatorname{dom}(\overline{\boldsymbol{\partial}})=\overline{\mathfrak{D}}_p(\Omega;\mathbb{B})$, and let $\overline{\boldsymbol{\partial}}-\mathbf{Q}_{(a,b)}: \overline{\mathfrak{D}}_p(\Omega;\mathbb{B})\rightarrow L_p(\Omega;\mathbb{B})$ be the sum of $\overline{\boldsymbol{\partial}}$ with the bounded operator $\mathbf{Q}_{(a,b)}$.

\begin{proposition}\label{propositionoperatorisclosed}
	The operator $\overline{\boldsymbol{\partial}}-\mathbf{Q}_{(a,b)}: \overline{\mathfrak{D}}_p(\Omega;\mathbb{B})\rightarrow L_p(\Omega;\mathbb{B})$ is closed.
\end{proposition}
\begin{proof}
	The case when $a\equiv b\equiv 0$ is given by Remark \ref{completenessweakderivative}. For the general case, consider $\{W_n\}\subset \overline{\mathfrak{D}}_p(\Omega;\mathbb{B})$ along with $W,V\in L_p(\Omega;\mathbb{B})$ such that $W_n\rightarrow W$ and $(\overline{\boldsymbol{\partial}}-\mathbf{Q}_{(a,b)})W_n\rightarrow V$ in $L_p(\Omega;\mathbb{B})$. Taking $H_n=\mathbf{S}_{\Omega}^{(a,b)}W_n$ and $H=\mathbf{S}_{\Omega}^{(a,b)}W$, the continuity of $\mathbf{S}_{\Omega}^{(a,b)}$ and relation \eqref{relationpartialS} imply that $H_n\rightarrow H$ and $\overline{\boldsymbol{\partial}}H_n=V_n\rightarrow V$ in $L_p(\Omega;\mathbb{B})$. Since $\overline{\boldsymbol{\partial}}$ is closed, $H\in \overline{\mathfrak{D}}_p(\Omega;\mathbb{B})$ and $\boldsymbol{\partial}H=V$. By the equality  $W=H+\mathbf{T}_{\Omega}\mathbf{Q}_{(a,b)}W\in \overline{\mathfrak{D}}_p(\Omega;\mathbb{B})$ and \eqref{relationpartialS}, the operator $\overline{\boldsymbol{\partial}}-\mathbf{Q}_{(a,b)}$ is closed.
\end{proof}

\begin{theorem}
	The Vekua-Bergman space $\mathcal{A}_{(a,b)}^p(\Omega; \mathbb{B})$ is a Banach space, and for $1<p<\infty$ is separable and reflexive. In particular, $\mathcal{A}_{(a,b)}^2(\Omega;\mathbb{B})$ is a separable Hilbert space.
\end{theorem}
\begin{proof}
	The $\mathcal{A}_{(a,b)}^p(\Omega;\mathbb{B})$ is a closed subspace of $L_p(\Omega;\mathbb{B})$ (and consequently, a Banach space) since it is the null space of the closed operator $\overline{\boldsymbol{\partial}}-\mathbf{Q}_{(a,b)}$. Since for $1<p<\infty$ the space $L_p(\Omega;\mathbb{B})$ is separable and reflexive, $\mathcal{A}_{(a,b)}^p(\Omega;\mathbb{B})$ is also separable and reflexive (see \cite{brezis}, Propositions 3.21 and 3.25). 
\end{proof}

\section{Regularity of the solutions}

In this section, we study the regularity of the solutions of the non-homogeneous Vekua equation
\begin{equation}\label{nonhomogeneousVekua}
	\left(\overline{\boldsymbol{\partial}}-\mathbf{Q}_{(a,b)}\right)W=V,\quad \mbox{with }\; V\in L_p(\Omega;\mathbb{B}), \; 1<p<\infty.
\end{equation}
The notation $G\Subset \Omega$ means that $G$ is an open subset of $\Omega$ with the property that $\overline{G}\subset \Omega$.
\begin{proposition}\label{Propestimatenorm}
	Let $W\in \overline{\mathfrak{D}}_p(\Omega;\mathbb{B})$ be a solution of (\ref{nonhomogeneousVekua}). For every $G\Subset \Omega$, there exists a constant $C_G>0$ (which does not depend on $W$) such that
	\begin{equation}\label{estimatenorm1}
		\|W|_D\|_{W^{1,p}(G;\mathbb{B})}\leqslant C_G\left(\|V\|_{L_p(\Omega;\mathbb{B})}+\|W\|_{L_p(\Omega;\mathbb{B})}\right).
	\end{equation}
\end{proposition}
\begin{proof}
	Let $B$ be an open set with $G\Subset B\Subset \Omega$ and $\eta\in C_0^{\infty}(\Omega;\mathbb{B})$ a cut-off function satisfying $\eta\equiv 1$ in $G$ and with $\operatorname{supp}\eta \subset B$. Hence $U=\eta W\in W_0^{1,p}(\Omega)$ (due to Proposition \ref{propositiondiferenciability}), and by Proposition \ref{remarktheodorescuop}(vi), $U=\mathbf{T}_{\Omega}\overline{\boldsymbol{\partial}}U$. Thus, 
	\[
	U=\mathbf{T}_{\Omega}\left[\eta\overline{\boldsymbol{\partial}}W+W\overline{\boldsymbol{\partial}}\eta\right]=\mathbf{T}_{\Omega}\left[\eta\mathbf{Q}_{a,b}W\right]+\mathbf{T}_{\Omega}\left[\eta V\right]+\mathbf{T}_{\Omega}\left[W\overline{\boldsymbol{\partial}}\eta\right].
	\]
	In particular, the right-hand side equals $W(z)$ for a.e. $z\in B$. Hence 
	\begin{align*}
		\|W|_D\|_{W^{1,p}(G;\mathbb{B})}& \leqslant \|\mathbf{T}_{\Omega}\left[\eta\mathbf{Q}_{a,b}W\right]\|_{W^{1,p}(\Omega;\mathbb{B})}+\|\mathbf{T}_{\Omega}\left[\eta V\right]\|_{W^{1,p}(\Omega;\mathbb{B})}+\|\mathbf{T}_{\Omega}\left[W\overline{\boldsymbol{\partial}}\eta\right]\|_{W^{1,p}(\Omega;\mathbb{B})}\\
		&\leqslant 2\operatorname{diam}(\Omega)\left[\|\eta\|_{L_{\infty}(\Omega;\mathbb{B})}\left(M_1\|W\|_{L_p(\Omega;\mathbb{B})}+\|V\|_{L_p(\Omega;\mathbb{B})}\right)+\|\overline{\boldsymbol{\partial}}\eta\|_{L_{\infty}(\Omega;\mathbb{B})}\|W\|_{L_p(\Omega;\mathbb{B})}\right],
	\end{align*}
	where $M_1=\|\mathbf{Q}_{(a,b)}\|_{\mathcal{B}(L_p(\Omega;\mathbb{B}))}$. Consequently, we obtain (\ref{estimatenorm1}) with\\ $C_G=2\operatorname{diam}(\Omega)\|\eta\|_{W^{1,\infty}(\Omega;\mathbb{B})}\max\{M_1,1\}$.
\end{proof}

\begin{proposition}\label{propregularityofsolutions}
	Let $W\in \overline{\mathfrak{D}}_p(\Omega;\mathbb{B})$ be a  solution of (\ref{nonhomogeneousVekua}) with $V\in L_q(\Omega; \mathbb{B})$ and $q\geqslant p$, $q>2$. The following statements hold.
	\begin{itemize}
		\item[(i)] If $p<2$, then $W\in C_{locdisk}^{0,1-\frac{2}{r}}(\Omega;\mathbb{B})$ for all $2<r<\min\{p^*,q\}$.
		\item[(ii)] If $p\geqslant 2$, then  $W\in C_{locdisk}^{0,1-\frac{2}{q}}(\Omega; \mathbb{B})$.
		\item[(iii)] $\mathcal{A}_{(a,b)}^p(\Omega;\mathbb{B})\subset C^{0,1-\frac{2}{r}}_{locdisk}(\Omega;\mathbb{B})$ for all $2<r<p^*$, if $p<2$, and for all $p<r<\infty$, if $p\geqslant 2$.
	\end{itemize}

\end{proposition}
\begin{proof}
	Take an arbitrary disk $D\Subset \Omega$. By \eqref{relationpartialS}, we obtain the equality
	\begin{equation*}
		W|_D=G+\mathbf{T}_D\left(\mathbf{Q}_{(a,b)}W|_D+V|_D\right)\quad \mbox{with } G\in \mathcal{A}^p(D;\mathbb{B}).
	\end{equation*}
	Since $G\in C^{\infty}(D;\mathbb{B})$, we focus on the integral $\mathbf{T}_D\left(\mathbf{Q}_{(a,b)}W|_D+V|_D\right)$.
	\begin{itemize}
		\item[(i)] When $1<p<2$, we know that $p^*>2$. Taking $2<r<\min\{p^*,q\}$, by Remark \ref{remarksobolevembedding} and Proposition \ref{propositiondiferenciability}(ii), $W|_D\in L_r(D;\mathbb{B})$, and hence $\mathbf{Q}_{(a,b)}W|_D+V|_D\in L_r(D;\mathbb{B})$. Due to Proposition \ref{remarktheodorescuop}(iii) and Remark \ref{remarksobolevembedding}, $\mathbf{T}_{D} \left(\mathbf{Q}_{(a,b)}W|_D+V|_D\right)\in C^{0,1-\frac{2}{r}}(\overline{D};\mathbb{B})$. By the arbitrariness of $D$, we obtain (i).
		\item[(ii)] In the case $p\geqslant 2$, by Remark \ref{remarksobolevembedding} and Proposition \ref{propositiondiferenciability}, $\mathbf{Q}_{(a,b)}W|_D+V|_D\in L_q(D;\mathbb{D})$. Again, by  Proposition \ref{remarktheodorescuop}(iii) and Remark \ref{remarksobolevembedding}, $\mathbf{T}_{D} \left(\mathbf{Q}_{(a,b)}W|_D+V|_D\right)\in C^{0,1-\frac{2}{q}}(\overline{D};\mathbb{B})$. Due to the arbitrariness of $D$, we conclude (ii).
		\item[(iii)] Taking $V\equiv 0$, this follows from points (i) and (ii).
	\end{itemize}
\end{proof}´

We recall that for a bounded domain $G\subset \mathbb{C}$, a function $W$ belongs to the class $C^{1,\epsilon}(\overline{G};\mathbb{B})$, $\epsilon\in (0,1]$, if $W\in C^1(\overline{G};\mathbb{B})$ and its partial derivatives belong to $C^{0,\epsilon}(\overline{G};\mathbb{B})$. We say that $W\in C^{1,\epsilon}_{locdisk}(\Omega;\mathbb{B})$ if $W|_D\in C^{1,\epsilon}(\overline{D};\mathbb{B})$ for any open disk $D\Subset \Omega$.

\begin{proposition}\label{propositionregularclassicalsol}
	If $a,b,V\in C^{0,\epsilon}_{locdisk}(\Omega;\mathbb{B})$, then every solution $W\in \overline{\mathfrak{D}}_2(\Omega;\mathbb{B})$ of (\ref{nonhomogeneousVekua}) belongs to $C^{1,\epsilon}_{locdisk}(\Omega;\mathbb{B})$. In particular $\mathcal{A}^p_{(a,b)}(\Omega;\mathbb{B})\subset C^{1,\epsilon}_{locdisk}(\Omega;\mathbb{B})$.
\end{proposition}
\begin{proof}
	Let $D\Subset \Omega$ be a disk, and write $W|_D=G+\mathbf{T}_D(\mathbf{Q}_{(a,b)}W|_D+V|_D)$ with $G\in \mathcal{A}^p(D;\mathbb{B})$. Since $G\in C^{\infty}(D)$, it only remains to analyze  $\mathbf{T}_{D}(\mathbf{Q}_{(a,b)}W|_D+V|_D)$. Using that for any disk $B$ with $D\Subset B\Subset \Omega$, $V|_B\in L_q(B;\mathbb{D})$, with $q=\frac{2}{1-\epsilon}>2$ , Proposition \ref{propregularityofsolutions}(ii) implies that $W|_D\in C^{0,\epsilon}(\overline{D};\mathbb{B})$. Since $a|_D,b|_D\in C^{0,\epsilon}(\overline{D};\mathbb{B})$, by \cite[Prop. 1.2.2]{holderfiorenza}, $\mathbf{Q}_{(a,b)}W|_D\in C^{0,\epsilon}(\overline{B};\mathbb{B})$. Consequently,  $\mathbf{T}_D(\mathbf{Q}_{(a,b)}W|_D+V|_D)\in C^{1,\epsilon}(\overline{D};\mathbb{B})$ \cite[Prop. 10 (v)]{CamposBicomplex}. Therefore $W\in C^{1,\epsilon}_{locdisk}(\Omega;\mathbb{B})$. 
\end{proof}
\begin{remark}\label{Remarkclassicalsolutions}
	For $a,b\in C^{0,\epsilon}_{locdisk}(\Omega;\mathbb{B})$, let $\mathcal{V}_{(a,b)}(\Omega;\mathbb{B})$ be the class of classical solutions $W\in C^1(\Omega;\mathbb{B})$ of (\ref{bicomplexvekua}). Proposition \ref{propositionregularclassicalsol} implies that $\mathcal{A}^p_{(a,b)}(\Omega;\mathbb{B})=\mathcal{V}_{(a,b)}(\Omega;\mathbb{B})\cap L_p(\Omega;\mathbb{B})$. In this case, the theory of bicomplex classical solutions of (\ref{bicomplexvekua}) was developed in \cite{CamposBicomplex}. The space $\mathcal{V}_{(a,b)}(\Omega;\mathbb{B})$ is closed in the Fr\'echet space $C(\Omega;\mathbb{B})$, with respect to the topology of the uniform convergence on compact subsets \cite[Th. 13]{CamposBicomplex}.
\end{remark}

By Proposition \ref{propregularityofsolutions}, every $W\in \mathcal{A}_{(a,b)}^p(\Omega;\mathbb{B})$ is continuous in $\Omega$ and for every $z\in \Omega$, the evaluation map $\mathcal{A}_{(a,b)}^2(\Omega;\mathbb{B})\ni W\mapsto W(z)\in \mathbb{B}$ is well-defined.

\begin{remark}\label{Remarkevalfunctionalclassicalbergman}
	Consider the case $a\equiv b\equiv 0$. Let $K\subset \Omega$ be compact. Given $W\in \mathcal{A}^p(\Omega;\mathbb{B})$, we have that $W^{+}\in \overline{\mathcal{A}}^p(\Omega)$ and $W^{-}\in \mathcal{A}^p(\Omega)$, the complex anti-analytic and analytic Bergman spaces, respectively. Then there exist constants $C_K^1$, $C_K^2$ depending only on $K$ such that 
	\[
	\max_{z\in K}|W^{-}(z)|\leqslant C_K^1\|W^{-}\|_{L_p(\Omega)} \quad \mbox{and }\quad \max_{z\in K}|W^{+}(z)|\leqslant C_K^2\|W^{+}\|_{L_p(\Omega)}
	\]
	(see \cite[Ch. I]{duren}). Thus, $\displaystyle \max_{z\in K}|W(z)|_{\mathbb{B}}\leqslant C_K\|W\|_{L_p(\Omega;\mathbb{B})}$ with $C_K=2\max\{C_K^1,C_K^2\}$.
\end{remark}

\begin{theorem}\label{theorembergmankernel}
	For every $K\subset \Omega$ compact, there exists a constant $C_K>0$ such that 
	\begin{equation}\label{boundednessofevalfunctional}
		\max_{z\in K}|W(z)|_{\mathbb{B}}\leqslant C_K\|W\|_{L_p(\Omega)},\quad \forall W\in \mathcal{A}_{(a,b)}^p(\Omega;\mathbb{B}).
	\end{equation}
\end{theorem}
\begin{proof}
	First, consider the case when $K=\overline{D}$ is a closed disk. Let $B$ be an open disk with $\overline{D}\subset B\Subset \Omega$, and take $H=\mathbf{S}_{B}^{(a,b)}W|_B$. Since $H\in \mathcal{A}^p(B;\mathbb{B})$, by Remark \ref{Remarkevalfunctionalclassicalbergman}, there exits a constant $C_{\overline{D}}^1$ satisfying 
	\[
	\max_{z\in \overline{D}}|H(z)| \leqslant C_{\overline{D}}^1\|H\|_{L_p(B;\mathbb{B})}= C_{\overline{D}}^1\|\mathbf{S}_{B}^{(a,b)}W|_B\|_{L_p(B;\mathbb{B})}\leqslant C_{\overline{D}}^1M_1\|W\|_{L_p(\Omega;\mathbb{B})},
	\]
	where $M_1=\|\mathbf{S}_B^{(a,b)}\|_{\mathcal{B}(L_p(B;\mathbb{B}))}$. Due to Remark \ref{remarksobolevembedding} and the fact that $p>1$, there exits $r>2$ such that the embeddings
	\begin{equation}\label{aux2}
		W^{1,p}(B;\mathbb{B})\hookrightarrow L_r(B;\mathbb{B})\quad \mbox{ and } \;\, W^{1,r}(B;\mathbb{B})\hookrightarrow C^{0,1-\frac{2}{r}}(\overline{B};\mathbb{B})
	\end{equation}  are bounded, and let $\tilde{C}_B^r$ and $\hat{C}_B^r$ be their norms. On the other hand, by Proposition \ref{propositionregularclassicalsol}, there is a constant $\tilde{C}_B^{p}>0$, independent of $W$, such that $\|W|_B\|_{W^{1,p}(B;\mathbb{B})}\leqslant \tilde{C}_B^p\|W\|_{L_p(\Omega;\mathbb{B})}$. Since $W|_B\in W^{1,p}(B;\mathbb{B})$, by Proposition \ref{propositiondiferenciability} and \eqref{aux2}, $V=\mathbf{T}_B\mathbf{Q}_{(a,b)}W|_B\in W^{1,r}(B;\mathbb{B})$. Taking $M_2=\|\mathbf{T}_B\mathbf{Q}_{(a,b)}\|_{\mathcal{B}(L_r(B;\mathbb{B}),W^{1,r}(B;\mathbb{B}))}$, we obtain
	\begin{align*}
		\max_{z\in \overline{D}}|V(z)|_{\mathbb{B}}&\leqslant \|V\|_{C^{0,1-\frac{2}{r}}(B;\mathbb{B})}\leqslant \hat{C}_B^r\|V\|_{W^{1,r}(B;\mathbb{B})}\leqslant \hat{C}_B^rM_2\|W\|_{L_r(B;\mathbb{B})}\\
		&\leqslant \hat{C}_B^rM_2\tilde{C}_B^r\|W\|_{W^{1,p}(B;\mathbb{B})}\leqslant \hat{C}_B^rM_2\tilde{C}_B^r\tilde{C}_B^p\|W\|_{L_p(\Omega;\mathbb{B})}.
	\end{align*}	
	Taking $C_{\overline{D}}=\max\{C_{\overline{D}}^1M_1,\hat{C}_B^rM_2\tilde{C}_B^r\tilde{C}_B^p\}$, we have that $\max\limits_{z\in \overline{D}}|W(z)|_{\mathbb{B}}\leqslant C_{\overline{D}}\|W\|_{L_p(\Omega;\mathbb{B})}$. For a general compact subset $K$, take $D_1,\dots, D_N$ open disks with $K\subset \bigcup_{j=1}^N\overline{D}_j\subset \Omega$. Hence inequality (\ref{boundednessofevalfunctional}) is satisfied by taking $C_K=\displaystyle \max_{1\leqslant j\leqslant N}C_{\overline{D}_j}$.
\end{proof}

\section{The Bergman reproducing kernel}
From now on, we focus on the case $p=2$. Theorem \ref{theorembergmankernel} implies that for any $z\in \Omega$, the functionals
\[
\mathcal{A}^2_{(a,b)}(\Omega;\mathbb{B})\ni W\mapsto \operatorname{Sc}W(z), \operatorname{Vec}W(z)\in \mathbb{C}
\] are bounded with respect to the $L_2(\Omega;\mathbb{B})$-norm. By the Riesz representation theorem, there exist functions $K_z,L_z\in \mathcal{A}_{(a,b)}^2(\Omega;\mathbb{B})$ such that
\begin{equation}\label{reproducingkernel1}
	\operatorname{Sc}W(z)=\langle W,K_z\rangle_{L_2(\Omega;\mathbb{B})}, \quad \operatorname{Vec}W(z)=\langle W,L_z\rangle_{L_2(\Omega;\mathbb{B})},\quad \forall W\in \mathcal{A}_{(a,b)}^2(\Omega;\mathbb{B}).
\end{equation}
\begin{remark}\label{remarkbergmankernel2}	
	The functions $K_z(\zeta)$ and $L_z(\zeta)$ satisfy the relations
	\begin{equation}\label{relationsbergmankernel}
		\operatorname{Sc}K_z(\zeta)=(\operatorname{Sc}K_{\zeta}(z))^{*}, \quad \operatorname{Vec}L_z(\zeta)=(\operatorname{Vec}L_{\zeta}(z))^{*}, \quad \operatorname{Sc}L_z(\zeta)= (\operatorname{Vec}K_{\zeta}(z))^{*}.
	\end{equation}
	Indeed, since $K_z\in \mathcal{A}_{(a,b)}^2(\Omega; \mathbb{B})$ we have
	\[
	\operatorname{Sc}K_z(\zeta)= \langle K_z, K_{\zeta}\rangle_{L_2(\Omega;\mathbb{B})}= \langle K_{\zeta}, K_z\rangle_{L_2(\Omega;\mathbb{B})}^{*}=\left(\operatorname{Sc}K_{\zeta}\right)^*(z).
	\]
	The other two equalities are proved analogously.
\end{remark}
Using (\ref{relationsbergmankernel}), for any $W\in \mathcal{A}_{(a,b)}^2(\Omega; \mathbb{B})$ we obtain  the following relation:
\begin{align*}
	W(z)& = \operatorname{Sc}W(z)+\mathbf{j}\operatorname{Vec}W(z)\\
	& = \iint_{\Omega}\left( \operatorname{Sc}W(\zeta)(\operatorname{Sc}K_z(\zeta))^{*}+\operatorname{Vec}W(\zeta)(\operatorname{Vec}K_z(\zeta))^{*}\right)dA_{\zeta} \\
	&\quad + \mathbf{j}\iint_{\Omega}\left( \operatorname{Sc}W(\zeta)(\operatorname{Sc}L_z(\zeta))^{*}+\operatorname{Vec}W(\zeta)(\operatorname{Vec}L_z(\zeta))^{*}\right)dA_{\zeta}\\
	& = \iint_{\Omega}\left( \operatorname{Sc}W(\zeta)\operatorname{Sc}K_{\zeta}(z)+\operatorname{Vec}W(\zeta)\operatorname{Sc}L_{\zeta}(z)\right)dA_{\zeta} \\
	&\quad + \mathbf{j}\iint_{\Omega}\left( \operatorname{Sc}W(\zeta)\operatorname{Vec}K_{\zeta}(z)+\operatorname{Vec}W(\zeta)\operatorname{Vec}L_{\zeta}(z)\right)dA_{\zeta},
\end{align*}
from where we obtain the relation
\begin{equation}\label{bergmankernel1}
	W(z)=\iint_{\Omega} \left( \operatorname{Sc}W(\zeta)K_{\zeta}(z)+\operatorname{Vec}W(\zeta)L_{\zeta}(z)\right)dA_{\zeta}.
\end{equation}
For all $z,\zeta\in \Omega$, define $K(z,\zeta)=K_{\zeta}(z)$ and $L(z,\zeta)=L_{\zeta}(z)$. We introduce the definition of the Bergman kernel.

\begin{definition}
	The Bergman kernel of the space $\mathcal{A}_{(a,b)}^2(\Omega; \mathbb{B})$ with coefficient $A\in \mathbb{B}$ is defined by
	\begin{equation}\label{Bergmankernelvekua}
		\mathscr{K}_{\Omega}^{(a,b)}(A;z,\zeta):= \operatorname{Sc}(A) K(z,\zeta)+\operatorname{Vec}(A)L(z,\zeta), \quad z,\zeta\in \Omega. 
	\end{equation}
\end{definition}
The following reproducing property holds:
\begin{equation}\label{reproducingproperty}
	W(z)=\iint_{\Omega}\mathscr{K}_{\Omega}^{(a,b)}(W(\zeta);z,\zeta)dA_{\zeta}.
\end{equation}
This definition for the bicomplex Bergman kernel of the Vekua equation was introduced in \cite{minevekua1} for the case of the space of classical $L_2$-solutions of the {\it main Vekua equation} (defined in Section \ref{sectionmainvekua}).

Because the space is separable, it admits a countable orthonormal basis $\{\Phi_n\}_{n=0}^{\infty}$ \cite[Ch. I, Prop. 4.16]{conwayfunctional}. Since $\mathscr{K}_{\Omega}^{(a,b)}(A; \cdot, \zeta)\in \mathcal{A}_{(a,b)}^2(\Omega; \mathbb{B})$ for all $A\in \mathbb{B}$, $\zeta \in \Omega$, it possesses a Fourier series in terms of $\{\Phi_n\}_{n=0}^{\infty}$, and the form of its Fourier coefficients can be obtained in terms of the coefficient $A$.

\begin{proposition}\label{propexpansionkernelfourier}
	For any $A\in \mathbb{B}$, the kernel $\mathscr{K}_{\Omega}^{(a,b)}(A;z,\zeta)$ admits the Fourier series
	\begin{equation}\label{bergmankernelexpansion}
		\mathscr{K}_{\Omega}^{(a,b)}(A; z, \zeta)=\sum_{n=0}^{\infty}\langle  A, \Phi_n(\zeta)\rangle_{\mathbb{B}}\Phi_n(z).
	\end{equation}
	The series converges in the variable $z$ in the $L_2(\Omega;\mathbb{B})$-norm and uniformly on compact subsets of $\Omega$.
\end{proposition}

\begin{proof}
	The proof of equality \eqref{bergmankernelexpansion} is the same as in \cite[Remark 17]{minevekua1}. The  uniform convergence on compact subsets of $\Omega$ in the variable $z$ follows from Theorem \ref{theorembergmankernel}.
\end{proof}

Since $\mathcal{A}_{(a,b)}^2(\Omega;\mathbb{B})$ is a closed subspace of $L_2(\Omega;\mathbb{B})$, there exists the bounded orthogonal projection of $L_2(\Omega;\mathbb{B})$ onto $\mathcal{A}_{(a,b)}^2(\Omega;\mathbb{B})$ \cite[Ch. I, Th. 2.7]{conwayfunctional}, denoted as $\mathbf{P}_{\Omega}^{(a,b)}$. We refer to it as the {\it Vekua-Bergman projection}. Similar to the case of the  bicomplex main Vekua equation \cite[Remark 18]{minevekua1}, the Vekua-Bergman projection can be written in terms of the kernel $\mathscr{K}_{\Omega}^{(a,b)}(A;z,\zeta)$. 

\begin{proposition}
	The Vekua-Bergman projection can be written as
	\begin{equation}\label{Vekua-Bergmanprojection}
		\mathbf{P}_{\Omega}^{(a,b)}\Psi(z)= \iint_{\Omega}\mathscr{K}_{\Omega}^{(a,b)}(\Psi(\zeta),z,\zeta)dA_{\zeta}, \qquad \forall \Psi \in L_2(\Omega;\mathbb{B}).
	\end{equation}
\end{proposition}
\begin{proof}
	The proof is the same as in \cite[Remark 18]{minevekua1}.
\end{proof}

In the case $b=0$, the Vekua-Bergman space $\mathcal{A}^2_{(a,0)}(\Omega;\mathbb{B})$ becomes a $\mathbb{B}$-module. Consequently, certain properties are analogous to those observed for the complex Bergman spaces, beginning with the behavior of the reproducing kernel itself.

\begin{proposition}
	When $b=0$, for every $A\in \mathbb{B}$ we have $\mathscr{K}_{\Omega}^{(a,0)}(A;z,\zeta)= AK_{\Omega}^a(z,\zeta)$, where $K_{\Omega}^a(z,\zeta)=K_{\zeta}(z)$ and $K_{\zeta}$ is defined  by (\ref{reproducingkernel1}). Thus, we have the reproducing property
	\begin{equation}\label{reproducingpropertiebzero}
		W(z)=\iint_{\Omega}K_{\Omega}^a(z,\zeta)W(\zeta)dA_{\zeta}\qquad \forall W\in \mathcal{A}_{(a,0)}^2(\Omega;\mathbb{B}).
	\end{equation}
\end{proposition}	

\begin{proof}
	Let $K(z,\zeta)=K_{\zeta}(z)$ and $L(z,\zeta)=L_{\zeta}(z)$ as in (\ref{reproducingkernel1}). Using the equalities $\operatorname{Sc}(\mathbf{j}W)=-\operatorname{Vec}W$, $\operatorname{Vec}(\mathbf{j}W)=\operatorname{Sc}W$, and \eqref{bicomplexinnerproductwithj}, we obtain
	\begin{align*}
		\operatorname{Sc}L(z,\zeta)& = \operatorname{Sc}L_{\zeta}(z)= \operatorname{Vec}\left(\mathbf{j}L_{\zeta}(z)\right) = \langle \mathbf{j}L_{\zeta}, L_z\rangle_{L_2(\Omega; \mathbb{B})}=-\langle L_{\zeta},\mathbf{j} L_z\rangle_{L_2(\Omega; \mathbb{B})}\\
		& = -\left(\operatorname{Vec}(\mathbf{j}L_z(\zeta))\right)^*=
		-\left(\operatorname{Sc}L(\zeta,z)\right)^{*}.
	\end{align*}
	Similarly, $\operatorname{Sc}K(z,\zeta)=\left(\operatorname{Vec}L(\zeta,z)\right)^{*}$. Using these equalities together with (\ref{relationsbergmankernel}), we obtain
	\begin{align*}
		L(z,\zeta) & = \operatorname{Sc}L(z,\zeta)+\mathbf{j}\operatorname{Vec}L(z,\zeta)=-\left(\operatorname{Sc}L(\zeta,z)\right)^{*}+\mathbf{j}\left(\operatorname{Sc}K(\zeta,z)\right)^{*}\\
		&= -\operatorname{Vec}K(z,\zeta)+\mathbf{j}\operatorname{Sc}K(z,\zeta)= \mathbf{j}K(z,\zeta).
	\end{align*}
	Substituting this equality into (\ref{Bergmankernelvekua}) and (\ref{reproducingproperty}), we obtain that $K_{\Omega}^{(a,0)}(A;z,\zeta)=AK(z,\zeta)$ for every $A\in \mathbb{B}$ and the reproducing property (\ref{reproducingpropertiebzero}).
\end{proof}	

\begin{remark}
	For the $\mathbb{B}$-analytic Bergman space $\mathcal{A}^2(\Omega;\mathbb{B})$, the Vekua-Bergman kernel is denoted by $\mathscr{K}_{\Omega}(z,\zeta)$. Using the relation $W^{\pm}(z)=\left(\iint_{\Omega}\mathscr{K}_{\Omega}(z,\zeta)W(\zeta)dA_{\zeta}\right)^{\pm}=\iint_{\Omega}\mathscr{K}^{\pm}(z,\zeta)W^{\pm}(\zeta)dA_{\zeta}$, we conclude that $\mathscr{K}^{\pm}_{\Omega}(z,\zeta)$ are the anti-analytic and the analytic complex Bergman kernels \cite[Ch.I]{duren}.
\end{remark}
\begin{remark}
	When $\mathcal{A}_{(a,b)}^2(\Omega;\mathbb{B})$ possesses a kernel $K_{\Omega}^{(a,b)}(z,\zeta)$ satisfying (\ref{reproducingpropertiebzero}), it necessarily follows that $K(z,\zeta)=K_{\Omega}^{(a,b)}(z,\zeta)$ and $L(z,\zeta)=\mathbf{j}K_{\Omega}^{(a,b)}(z,\zeta)$. Consequently,  $\mathscr{K}_{\Omega}^{(a,b)}(A;z,\zeta)=AK_{\Omega}^{(a,b)}(z,\zeta)$. Thus, the Vekua-Bergman projections takes the form\\ $\mathbf{P}_{\Omega}^{(a,b)}\Psi(z)=\iint_{\Omega}K_{\Omega}^{(a,b)}(z,\zeta)\Psi(\zeta)dA_{\zeta}$ for all $\Psi\in L_2(\Omega;\mathbb{B})$. For every $W\in \mathcal{A}_{(a,b)}^2(\Omega;\mathbb{B})$, we have $\mathbf{j}W(z)=\mathbf{j}\iint_{\Omega}K_{\Omega}^{(a,b)}(z,\zeta)W(\zeta)dA_{\zeta}=\iint_{\Omega}K_{\Omega}^{(a,b)}(z,\zeta)\mathbf{j}W(\zeta)dA_{\zeta}=\mathbf{P}_{\Omega}^{(a,b)}[\mathbf{j}W](z)$, hence $\mathbf{j}W\in \mathcal{A}_{(a,b)}^2(\Omega;\mathbb{B})$ and the Vekua-Bergman space is a $\mathbb{B}$-module. Substituting $\mathbf{j}W$ into (\ref{bicomplexvekua}), we obtain the condition $b\overline{W}=0$ for all $W\in \mathcal{A}_{(a,b)}^2(\Omega;\mathbb{B})$. In particular, when $b$ is a constant belonging to $\mathcal{R}(\mathbb{B})$, such a kernel $K_{\Omega}^{(a,b)}(z,\zeta)$ cannot exists.
\end{remark}

\section{Hodge decomposition}
In this section, we assume that $\Gamma$ is a Liapunov curve.
\begin{proposition}\label{propclosedpartial}
	The adjoint of the operator $\overline{\boldsymbol{\partial}}$ is given by $(\overline{\boldsymbol{\partial}})^*= -\boldsymbol{\partial}: W_0^{1,2}(\Omega; \mathbb{B}) \rightarrow L_2(\Omega; \mathbb{B})$. 
\end{proposition}

\begin{proof}
	From the definition of the weak $\overline{\boldsymbol{\partial}}$-derivative, it is clear that $W_0^{1,2}(\Omega; \mathbb{B})\subset \operatorname{dom}(\overline{\boldsymbol{\partial}})^*$ and $(\overline{\boldsymbol{\partial}})^*|_{W_0^{1,2}(\Omega; \mathbb{B})}= -\boldsymbol{\partial}$. Take $V\in \operatorname{dom}(\overline{\boldsymbol{\partial}})^*$ and $\Psi\in \overline{\mathfrak{D}}_2(\Omega; \mathbb{B})$. Using the fact that this space is a $\mathbb{B}$-module, the $\mathbb{B}$-linearity of $\boldsymbol{\partial}$ and $\overline{\boldsymbol{\partial}}$ (Remark \ref{completenessweakderivative}(ii)), and equalities \eqref{bicomplexinnerproduct} and \eqref{bicomplexinnerproductwithj}, we obtain 
	\begin{align*}
		\iint_{\Omega}\overline{\boldsymbol{\partial}}\Psi V^{\dagger} & = \langle \overline{\boldsymbol{\partial}} \Psi,V\rangle_{L_2(\Omega; \mathbb{B})}+\mathbf{j}\langle \overline{\boldsymbol{\partial}}\Psi,\mathbf{j}V\rangle\\
		& = \langle \Psi,(\overline{\boldsymbol{\partial}})^*V\rangle_{L_2(\Omega; \mathbb{B})}-\mathbf{j}\langle \mathbf{j}\overline{\boldsymbol{\partial}}\Psi, V\rangle_{L_2(\Omega; \mathbb{B})} \\
		& = \langle \Psi,(\overline{\boldsymbol{\partial}})^*V\rangle_{L_2(\Omega; \mathbb{B})}-\mathbf{j}\langle \mathbf{j}\Psi,\left(\overline{\boldsymbol{\partial}}\right)^*V \rangle_{L_2(\Omega; \mathbb{B})} = \iint_{\Omega} \Psi ((\overline{\boldsymbol{\partial}})^*V)^{\dagger}
	\end{align*}
	Applying the involution $^{\dagger}$ to both sides of the equality and using (\ref{involutuionofcauchyriemannoperator})  we obtain 
	\begin{equation}\label{auxiliar1}
		\iint_{\Omega}\boldsymbol{\partial}\Psi V=\iint_{\Omega} \Psi (\overline{\boldsymbol{\partial}})^*V\quad \mbox{ for all } \Psi \in \overline{\mathfrak{D}}_2(\Omega; \mathbb{B}).
	\end{equation}
	In particular, this is valid for $\Psi\in C_0^{\infty}(\Omega; \mathbb{B})$, hence $-\boldsymbol{\partial}V=\left(\boldsymbol{\partial}\right)^*V$. By Proposition \ref{remarktheodorescuop}(vii), $V=\mathbf{T}_{\Omega}^*\left(\boldsymbol{\partial}\right)^*V+G$, with $G\in L_2(\Omega; \mathbb{B})$ a $\mathbb{B}$-anti-analytic function. Since $\mathbf{T}_{\Omega}^*\left(\overline{\boldsymbol{\partial}}\right)^*\subset \left(\overline{\boldsymbol{\partial}}\mathbf{T}_{\Omega}\right)^*=\mathbf{I}$ (by Proposition \ref{remarktheodorescuop}(ii) and \cite[Prop.1.7]{schmudgen}), $V=V+G$, and $G=0$. Thus, $V=\mathbf{T}_{\Omega}^*\left(\boldsymbol{\partial}\right)^*V\in W^{1,2}(\Omega; \mathbb{B})$. Finally, given $\psi\in W^{1,\frac{1}{2}}(\Gamma; \mathbb{B})$, let $\Phi\in W^{1,2}(\Omega; \mathbb{B})$ such that $\operatorname{tr}_{\Gamma}\Phi =\psi$. By the Green identities \eqref{bicomplexGreenidentities}  and \eqref{auxiliar1}, we have
	\[
	\frac{-1}{2\pi \mathbf{j}}\int_{\Gamma}\psi(\zeta)^{\dagger}\operatorname{tr}_{\Gamma}V(\zeta)d\zeta^{\dagger}= \iint_{\Omega} \boldsymbol{\partial}\Phi^{\dagger} V +\iint_{\Omega}\Phi^{\dagger} \boldsymbol{\partial}V=0,
	\]
	which implies $\int_{\Gamma}\psi(\zeta)\operatorname{tr}_{\Gamma}V^{\dagger}(\zeta)d\zeta=0$ for all $\psi\in W^{1,\frac{1}{2}}(\Omega; \mathbb{B})$. Lemma \ref{lemalineintegral} implies that $V=W_0+G^{\dagger}$, with $W_0\in W^{1,2}_0(\Omega; \mathbb{B})$ and $G\in \mathcal{A}^{2}(\Omega; \mathbb{B})\cap W^{1,2}(\Omega; \mathbb{B})$. Hence, for every $\Phi\in W^{1,2}(\Omega; \mathbb{B})$, from (\ref{auxiliar1}), we derive 
	\[
	\iint_{\Omega}\boldsymbol{\partial}\Phi(W_0+G^{\dagger})= -\iint_{\Omega}\Phi\boldsymbol{\partial}W_0, \quad \mbox{ and thus,}\; \iint_{\Omega}G^{\dagger}\boldsymbol{\partial}\Phi=0.
	\]
	By choosing $\Phi=-\mathbf{T}^{*}_{\Omega}G$, we deduce that $\iint_{\Omega}G^{\dagger}G=0$, and consequently $G=0$. Therefore $V=W_0\in W_0^{1,2}(\Omega; \mathbb{B})$ and $(\boldsymbol{\partial})^*=-\boldsymbol{\partial}$.
\end{proof}
\begin{remark}\label{remarkclosedoperators}
	Let $\mathbf{A}:\operatorname{dom}(\mathbf{A})\subset \mathcal{H}\rightarrow \mathcal{H}$ be a densely defined operator in a Hilbert space $\mathcal{H}$. When $\mathcal{A}$ is closed, according to \cite[Prop. 16(ii) and Th. 1.8 (iii)]{schmudgen}, $\ker \mathbf{A}$ is a closed subspace, $\mathbf{A}^{**}=\mathbf{A}$ and $\ker(\mathbf{A})=(\operatorname{Im}\mathbf{A})^{\perp}$. Consequently, this yields the decomposition 
	\begin{equation}\label{decompositionclosedoperator}
		\mathcal{H}=\ker \mathbf{A}\oplus \overline{\operatorname{Im}(\mathbf{A}^*)}.
	\end{equation}
\end{remark}

\begin{corollary}[Bicomplex Hodge decomposition]\label{bicomplexhodgecoro}
	The following equality holds
	\begin{equation}\label{hodgedecompositioneq}
		L_2(\Omega; \mathbb{B})= \mathcal{A}^2(\Omega; \mathbb{B})\oplus \boldsymbol{\partial}W_0^{1,2}(\Omega; \mathbb{B}).
	\end{equation}
\end{corollary}
\begin{proof}
	Since $\overline{\boldsymbol{\partial}}$ is closed, by Remark \ref{remarkclosedoperators}, we obtain the decomposition $L_2(\Omega; \mathbb{B})=\ker(\overline{\boldsymbol{\partial}})\oplus \overline{\operatorname{Im}(\overline{\boldsymbol{\partial}})^*}=\mathcal{A}^2(\Omega; \mathbb{B})\oplus \overline{\boldsymbol{\partial}W_0^{1,2}(\Omega; \mathbb{B})}$. It remains to show that $\boldsymbol{\partial}W_0^{1,2}(\Omega; \mathbb{B})$ is closed. By the Maximum principle, if $W\in W_0^{1,2}(\Omega; \mathbb{B})$ and $\overline{\boldsymbol{\partial}}W=0$, then $W\equiv 0$. Furthermore, according to Prop. \ref{remarktheodorescuop}(vii), the bounded operator $\mathbf{T}_{\Omega}^{*}$ is an extension of the inverse of $-\boldsymbol{\partial}$. Thus, $0$ is a regular point of $-\boldsymbol{\partial}$ and, consequently, its range is closed \cite[Prop. 2.1 (iv)]{schmudgen}. Therefore (\ref{hodgedecompositioneq}) follows.
\end{proof}

\begin{proposition}\label{propositionabjointvekua}
	The adjoint of $\overline{\boldsymbol{\partial}}-\mathbf{Q}_{(a,b)}$ is  $-\boldsymbol{\partial}-a^{\dagger}-b^*\mathbf{C}_{\mathbb{B}}: W_0^{1,2}(\Omega; \mathbb{B})\rightarrow L_2(\Omega; \mathbb{B})$, where $b^*:= (\operatorname{Sc}b)^*+\mathbf{j}(\operatorname{Vec}b)^*$. 
\end{proposition}
\begin{proof}
	Since $\overline{\boldsymbol{\partial}}-\mathbf{Q}_{(a,b)}$ is the difference between $\overline{\boldsymbol{\partial}}$ and the bounded operator $\mathbf{Q}_{(a,b)}$, its adjoint is given by  $\left(\overline{\boldsymbol{\partial}}-\mathbf{Q}_{(a,b)}\right)^*= (\overline{\boldsymbol{\partial}})^*-(a\mathbf{I})^*-(b\mathbf{C}_{\mathbb{B}})^*$ with domain $\operatorname{dom}(\overline{\boldsymbol{\partial}})^*=W_0^{1,2}(\Omega; \mathbb{B})$ (see \cite[Prop. 1.6]{schmudgen}). A direct computation shows that $(a\mathbf{I})^*=a^{\dagger}\mathbf{I}$ and $(b\mathbf{C}_{\mathbb{B}})^*=b^*\mathbf{C}_{\mathbb{B}}$, with $b^*=(\operatorname{Sc}b)^*+\mathbf{j}(\operatorname{Vec}b)^*$. 
\end{proof}

\begin{proposition}\label{remarkorthogonaldecompositionvekua}
	The following orthogonal decomposition holds:
	\begin{equation}\label{orthogonaldecompositionvekua}
		L_2(\Omega;\mathbb{B})=\mathcal{A}_{(a,b)}^2(\Omega;\mathbb{B})\oplus\overline{\left(\boldsymbol{\partial}+a^{\dagger}+b^*\mathbf{C}_{\mathbb{B}}\right)W_0^{1,2}(\Omega;\mathbb{B})}^{L_2(\Omega;\mathbb{B})}.
	\end{equation}
\end{proposition}
\begin{proof}
	It follows from Propositions \ref{propclosedpartial} and  \ref{propositionabjointvekua}, and Remark \ref{remarkclosedoperators}.
\end{proof}
\begin{remark}\label{remarkinvertibilityS}
	By \eqref{normoperatorQ},  $\|\mathbf{T}_{\Omega}\mathbf{Q}_{(a,b)}\|_{\mathcal{B}(L_2(\Omega;\mathbb{B}))}\leqslant 2\sqrt{2}\operatorname{diam}(\Omega)\max\{\|a\|_{L_{\infty}(\Omega;\mathbb{B})},\|b\|_{L_{\infty}(\Omega;\mathbb{B})}\}$. In particular, when $\max\{\|a\|_{L_{\infty}(\Omega;\mathbb{B})},\|b\|_{L_{\infty}(\Omega;\mathbb{B})}\}<\frac{1}{2\sqrt{2}\operatorname{diam}(\Omega)}$, $\mathbf{S}_{\Omega}^{(a,b)}\in \mathcal{G}\left(L_2(\Omega;\mathbb{B})\right)$ \cite[Th.  10.23]{axleranalysis}. By the same reasoning the operator $\tilde{\mathcal{S}}_{\Omega}^{(a,b)}:=\mathbf{I}- \mathbf{T}^{*}\mathbf{Q}_{(a,b)}^*$ is invertible.
\end{remark}
\begin{remark}\label{remarkdirichletproblem}
	Let $D_{(a,b)}$ and $D_{(a,b)}^*$ denote the solution spaces of the Dirichlet problems $(\overline{\boldsymbol{\partial}}-\mathbf{Q}_{(a,b)})W=0$ and $\left(\boldsymbol{\partial}+\mathbf{Q}_{(a,b)}^*\right)W=0$, respectively, where $W\in W^{1,2}_0(\Omega;\mathbb{B})$.
	Thus, $D_{(a,b)}\subset \ker\mathbf{S}_{\Omega}^{(a,b)}$ and $D_{(a,b)}^*\subset \ker\tilde{\mathcal{S}}_{\Omega}^{(a,b)}$. In particular, both spaces have finite dimensions.
	
	We establish the inclusion $D_{(a,b)}\subset \ker\mathbf{S}_{\Omega}^{(a,b)}$ (the proof of $D_{(a,b)}^*\subset \ker\tilde{\mathcal{S}}_{\Omega}^{(a,b)}$ is analogous). Let $W\in D_{(a,b)}$. By the Borel-Pompeiu formula we have $W=\mathbf{T}\overline{\boldsymbol{\partial}}W=\mathbf{T}\mathbf{Q}_{(a,b)}W$, hence $\mathbf{S}_{\Omega}^{(a,b)}W=0$, and consequently, $W\in \ker\mathbf{S}_{\Omega}^{(a,b)}$.
\end{remark}
\begin{theorem}[Hodge decomposition]
	If $\max\{\|a\|_{L_{\infty}(\Omega;\mathbb{B})},\|b\|_{L_{\infty}(\Omega;\mathbb{B})}\}<\frac{1}{2\sqrt{2}\operatorname{diam}(\Omega)}$, then the following decomposition holds
	\begin{equation}\label{hodgedecompositionVekua}
		L_2(\Omega;\mathbb{B})=\mathcal{A}_{(a,b)}^2(\Omega;\mathbb{B})\oplus \left(\boldsymbol{\partial}+a^{\dagger}+b^*\mathbf{C}_{\mathbb{B}}\right)W_0^{1,2}(\Omega;\mathbb{B}).
	\end{equation}
\end{theorem}
\begin{proof}
	According to Remark \ref{remarkinvertibilityS}, $\tilde{\mathcal{S}}_{\Omega}^{(a,b)}\in \mathcal{G}(L_2(\Omega;\mathbb{B}))$. As stated in Remark \ref{remarkdirichletproblem}, $0$ is not an eigenvalue of the operator $\boldsymbol{\partial}+\mathbf{Q}_{(a,b)}^*:W_0^{1,2}(\Omega;\mathbb{B})\rightarrow L_2(\Omega;\mathbb{B})$. Moreover, the inverse of $-(\boldsymbol{\partial}+\mathbf{Q}_{(a,b)}^*)$ is $\mathbf{R}=\left(\tilde{S}_{\Omega}^{(a,b)}\right)^{-1}\mathbf{T}_{\Omega}^*$. Indeed, since $\tilde{S}_{\Omega}^{(a,b)}(\mathfrak{D}_2(\Omega;\mathbb{B}))\subset \mathfrak{D}_2(\Omega;\mathbb{B})$ and $-\boldsymbol{\partial}\tilde{S}_{\Omega}^{(a,b)}W=-(\boldsymbol{\partial}+\mathbf{Q}_{(a,b)}^*)W$ for all $W\in \mathfrak{D}_2(\Omega;\mathbb{B})$, we have that $\boldsymbol{\partial}V=-(\boldsymbol{\partial}+\mathbf{Q}_{(a,b)}^*)\left(\tilde{\mathcal{S}}_{\Omega}^{(a,b)}\right)^{-1}V$ for all $V\in \mathfrak{D}_2(\Omega;\mathbb{B})$. Thus, when $W\in W_0^{1,2}(\Omega;\mathbb{B})$, using Proposition \ref{remarktheodorescuop}(vii), we obtain
	\[
	-\mathbf{R}\left(\boldsymbol{\partial}+\mathbf{Q}_{(a,b)}^*\right)W=\left(\tilde{\mathcal{S}}_{\Omega}^{(a,b)}\right)^{-1}\left(-\mathbf{T}_{\Omega}^*\boldsymbol{\partial}W-\mathbf{T}_{\Omega}^*\mathbf{Q}_{(a,b)}^*W\right)=\left(\tilde{\mathcal{S}}_{\Omega}^{(a,b)}\right)^{-1}\tilde{\mathcal{S}}_{\Omega}^{(a,b)}W=W
	\] 
	 Again, by Proposition \ref{remarktheodorescuop}(vii), given $V\in L_2(\Omega;\mathbb{B})$, we have
	\[
	-(\boldsymbol{\partial}+\mathbf{Q}_{(a,b)}^*)\mathbf{R}V=-(\boldsymbol{\partial}+\mathbf{Q}_{(a,b)}^*)\left(\tilde{\mathcal{S}}_{\Omega}^{(a,b)}\right)^{-1}\mathbf{T}_{\Omega}^*V=-\boldsymbol{\partial}\mathbf{T}^*V=V.
	\]	
	Therefore $-(\boldsymbol{\partial}+\mathbf{Q}_{(a,b)}^*)$ has a bounded inverse and $0$ is a regular value, implying  that its range is closed \cite[Prop. 2.1 (iv)]{schmudgen}, yielding (\ref{hodgedecompositionVekua}).
\end{proof}
\begin{remark}\label{remarkinversecaseb0}
	In \cite{CamposBicomplex}, the bicomplex exponential of $W\in \mathbb{B}$ is defined as $e^W:=\mathbf{p}^+e^{W^+}+\mathbf{p}^{-}e^{W^{-}}$. For every $W,V\in \mathbb{B}$, $e^{V+W}=e^{W}e^{V}$, and, in particular, $e^W\in \mathcal{R}(\mathbb{B})$ with $\left(e^{W}\right)^{-1}=e^{-W}$ (for the proof of these facts, see \cite[Prop. 4]{CamposBicomplex}). Note that $e^{\widehat{z}}$ is $\mathbb{B}$-analytic. A direct computation shows that for every $V\in W^{1,2}(\Omega;\mathbb{B})$, $\overline{\boldsymbol{\partial}}e^{V}=\overline{\boldsymbol{\partial}}V\cdot e^V$. Thus, $\Phi_a:= e^{\mathbf{T}_{\Omega}a}$ is a particular solution of $\left(\overline{\boldsymbol{\partial}}-a\right)W=0$.  Since $a\in L_{\infty}(\Omega;\mathbb{B})$, $\Phi_a\in C^{0,\epsilon}(\overline{\Omega}; \mathbb{B})\cap W^{1,2}(\Omega;\mathbb{B})$ for all $0<\epsilon<1$. A right-inverse operator for $\overline{\boldsymbol{\partial}}-a$ is given by $\mathbf{R}_{a,\Omega}V:=\Phi_a\mathbf{T}_{\Omega}\Phi_{-a}V$. Furthermore, every $W\in \mathcal{A}_{(a,0)}^2(\Omega;\mathbb{B})$ can be written as $W=\Psi \Phi_{a}$, with $\Psi=W\cdot \Phi_{-a}\in \mathcal{A}^2(\Omega;\mathbb{B})$. This is a version of the similarity principle for bicomplex pseudoanalytic functions \cite[Th. 14]{CamposBicomplex}. Actually, the operator $\mathbf{R}_a:= \Phi_a\mathbf{T}_{\Omega}\Phi_{-a}$  is not only a bounded right-inverse for $\left(\overline{\boldsymbol{\partial}}-a\right):W_0^{1,2}(\Omega;\mathbb{B})\rightarrow L_2(\Omega;\mathbb{B})$ but also a left-inverse:
	\begin{align*}
		\mathbf{R}_a\left(\overline{\boldsymbol{\partial}}-a\right)W & =\Phi_a\mathbf{T}_{\Omega}\Phi_{-a}\left(\overline{\boldsymbol{\partial}}W-aW\right)=\Phi_a\mathbf{T}_{\Omega}\left(\Phi_{-a}\overline{\boldsymbol{\partial}}W-a\Phi_{-a}W\right)\\
		&= \Phi_a\mathbf{T}_{\Omega}\overline{\boldsymbol{\partial}}(\Phi_{-a}W)= \Phi_a\Phi_{-a}W=W.
	\end{align*}
Hence $0$ is a regular value for $(\overline{\boldsymbol{\partial}}-a)$. Similarly, the operator $\tilde{\mathbf{R}}_a:= \Psi_{-a^{\dagger}}\mathbf{T}_{\Omega}^*\Psi_{a^{\dagger}}$, where $\Psi_a:=e^{-\mathbf{T}^*_{\Omega}a}$,  is the bounded inverse of $-\left(\boldsymbol{\partial}+a^{\dagger}\right):W_0^{1,2}(\Omega;\mathbb{B})\rightarrow L_2(\Omega;\mathbb{B})$. Thus, $0$ is a regular value, and $\left(\boldsymbol{\partial}+a^{\dagger}\right)W_0^{1,2}(\Omega;\mathbb{B})$ is closed. Consequently, the Hodge decomposition \eqref{hodgedecompositionVekua} is still valid for all $a\in L_{\infty}(\Omega;\mathbb{B})$ when $b\equiv0$.
\end{remark}

\section{Properties of the main Vekua equation and its relations with the conductivity equation}\label{sectionmainvekua}

Let $1<p<\infty$. In this section, we focus on the main Vekua Equation, corresponding to the coefficients $a\equiv 0$ and $b=\frac{\overline{\boldsymbol{\partial}}f}{f}$, where $f\in W^{1,\infty}(\Omega)$ is a non-vanishing complex scalar function. The main Vekua equation has the form
\begin{equation}\label{mainvekuaeq}
	\overline{\boldsymbol{\partial}}W-\dfrac{\overline{\boldsymbol{\partial}}f}{f} \overline{W}=0\quad \mbox{in }\Omega.
\end{equation}
 Let $\mathcal{A}_f^p(\Omega;\mathbb{B})$ denote the associated Vekua-Bergman space. To ensure $\dfrac{\overline{\boldsymbol{\partial}}f}{f}\in L_{\infty}(\Omega;\mathbb{B})$, we assume $\frac{1}{f}\in L_{\infty}(\Omega)$. In certain contexts, an $f$ satisfying these conditions is called a {\it proper conductivity}.

Let $W\in \mathcal{A}_f^p(\Omega;\mathbb{B})$. Since $W\in W^{1,p}_{loc}(\Omega;\mathbb{B})$ (by Proposition \ref{propositiondiferenciability}), a direct computation shows that $u=\operatorname{Sc}W$ and $v=\operatorname{Vec}W
$ satisfy (a.e. in $\Omega$ ) the system
\begin{align}
	u_x-v_y & =f_xu+f_yv, \label{system11}\\
    u_y+v_x & = f_yu-f_xv, \label{system12}	
\end{align}
or equivalently,
\begin{align}
f\frac{\partial}{\partial x}\left(\frac{u}{f}\right) & =	\frac{1}{f}\frac{\partial}{\partial y}(fv), \label{system21}\\
f\frac{\partial}{\partial y}\left(\frac{u}{f}\right) & =	-\frac{1}{f}\frac{\partial}{\partial x}(fv), \label{system22}	
\end{align}

Analogous to classical solutions where $f\in C^2(\Omega)$ (see \cite{hugo} or \cite[Ch. 3]{kravpseudo1}), the scalar and vector parts of $W$ satisfy a pair of conductivity equations.

\begin{proposition}\label{propscalarandvectorialconductivity}
 For all $W\in \mathcal{A}_f^p(\Omega;\mathbb{B})$, the function $U=\frac{u}{f}$, where $u= \operatorname{Sc}W$, is a weak solution in $W^{1,p}_{loc}(\Omega)$ of the conductivity equation
	\begin{equation}\label{conductivityequation1}
		\operatorname{div}\left(f^2\nabla U\right)=0\quad \mbox{ in } \Omega.
	\end{equation} 
Similarly, $V=fv$, where $v=\operatorname{Vec}W$ is a a weak solution in $W^{1,p}_{loc}(\Omega)$ of
\begin{equation}\label{conductivityequation2}
	\operatorname{div}\left(\frac{1}{f^2}\nabla V\right)=0\quad \mbox{ in } \Omega.
\end{equation} 
\end{proposition}
\begin{proof}
	Since $u\in W^{1,p}_{loc}(\Omega;\mathbb{B})$ and $\frac{1}{f}\in W^{1,\infty}(\Omega)$, we have $U=\frac{u}{f}\in W^{1,p}_{loc}(\Omega)$. Given $\varphi\in C_{0}^{\infty}(\Omega)$, equations \eqref{system21} and \eqref{system22} yield
	\begin{align*}
		\iint_{\Omega}f^2\nabla U\cdot \nabla V\varphi &= \iint_{\Omega}f\left\{f\frac{\partial}{\partial x}\left(\frac{u}{f}\right)\frac{\partial \varphi}{\partial x}+f\frac{\partial}{\partial y}\left(\frac{u}{f}\right)\frac{\partial \varphi}{\partial y}\right\}\\
		& = \iint_{\Omega}f\left\{\frac{1}{f}\frac{\partial (vf)}{\partial y}\frac{\partial \varphi}{\partial x}-\frac{1}{f}\frac{\partial(vf)}{\partial y}\frac{\partial \varphi}{\partial y}\right\}= -\iint_{\Omega}vf\left\{\frac{\partial^2\varphi}{\partial y\partial x}-\frac{\partial^2\varphi}{\partial x\partial y}\right\}=0.
	\end{align*}
Thus, $U$ is a weak solution of \eqref{conductivityequation1}. The proof of \eqref{conductivityequation2} is analogous.
\end{proof}

We will require the following lemma. 

\begin{lemma}\label{lemaregularidadconductivyeq}
	Let $G\subset \mathbb{C}$ be a bounded domain, and $\sigma \in W^{1,\infty}(G)$ such that $\frac{1}{\sigma}\in L_{\infty}(G)$. Denote the space of weak solutions in $W^{1,p}(G)$ of the conductivity equation $\operatorname{div}\sigma \nabla u=0$ by
	\begin{equation}\label{spaceofweaksolutionsconductivity}
		\operatorname{Sol}_{\sigma}^p(G):=\left\{u\in W^{1,p}(G)\, \Big{|} \forall \varphi\in C_0^{\infty}(G)\; \int_G \sigma \nabla u\cdot \nabla \varphi=0 \right\}.
	\end{equation}
The following statements hold:
\begin{itemize}
	\item[(i)] $\operatorname{Sol}_{\sigma}^p(G)$ is a closed subspace of $W^{1,p}(G)$.
	\item[(ii)] $\operatorname{Sol}_{\sigma}^p(G)\subset W^{2,q}_{loc}(G)$,  where $p\leqslant q\leqslant p^*$ if $1<p<2$, and $p\leqslant q< \infty$ if $p\geqslant 2$.
	\item[(iii)] $\operatorname{Sol}_{\sigma}^p(G)\subset C^{1,1-\frac{2}{r}}_{locdisk}(G)$, where $2<r<p^*$ if $1<p<2$, and $r>2$ if $p\geqslant2$. In particular, $\operatorname{Sol}_{\sigma}^p(G)\subset C^1(G)$.
	\item[(iv)] For every $G'\Subset G$, there exists $C_{G'}>0$ such that 
	\begin{equation}\label{inequalitylemmacondceq}
		\|u\|_{C^1(G')}\leqslant C_{G'}\|u\|_{W^{1,p}(G)},\quad \forall u\in \operatorname{Sol}_{\sigma}^p(G).
	\end{equation}   
\end{itemize}
\end{lemma}
The proof is given in Appendix A.

\begin{corollary}
	If $f\in C^1(\Omega)\cap W^{1,\infty}(\Omega)$, then $\mathcal{A}_f^p(\Omega;\mathbb{B})\subset C^1(\Omega;\mathbb{B})$, that is, every weak solution of the main Vekua equation is a classical solution.
\end{corollary}

\begin{proof}
	Let $W\in \mathcal{A}_f^p(\Omega;\mathbb{B})$ and $u=\operatorname{Sc}W$, $v=\operatorname{Vec}W$. Take $G\Subset \Omega$. By Propositions \ref{propregularityofsolutions} and \ref{propscalarandvectorialconductivity}, $U=\frac{u}{f}\in \operatorname{Sol}_{f^2}^p(G)$ and $V=fv\in \operatorname{Sol}_{\frac{1}{f^2}}^p(G)$. It follows from Lemma \ref{lemaregularidadconductivyeq}(iii) that $U,V\in C^1(G)$. Since $f\in C^1(\Omega)$, $u=fU, v=\frac{V}{f}\in C^1(G)$. Due to the arbitrariness of $G$, we conclude that $u,v\in C^1(\Omega)$.
\end{proof}

The following result generalizes a well-known property for classical solutions of \eqref{mainvekuaeq} when $f\in C^2(\Omega)$ (see \cite{CamposBicomplex, pseudoanalyticvlad}).
\begin{corollary}
	Suppose that $f\in W^{2,\infty}(\Omega)$. For every $W\in \mathcal{A}_f^p(\Omega;\mathbb{B})$, $u=\operatorname{Sc}W\in W^{2,p}_{loc}(\Omega)$ satisfies the Schr\"odinger equation 
	\begin{equation}\label{eqSchrodinger1}
		-\bigtriangleup u+q_fu=0, \quad \mbox{in } \Omega,
	\end{equation}
	with potential $q_f:=\frac{\bigtriangleup f}{f}$, and $v=\operatorname{Vec}W\in W^{2,p}_{loc}(\Omega)$ satisfies de Darboux-associated equation
	\begin{equation}\label{eqSchrodingerDarboux}
		-\bigtriangleup u+q_{\frac{1}{f}}u=0, \quad \mbox{in } \Omega,
	\end{equation}
	where $q_f=f\bigtriangleup \left(\frac{1}{f}\right)$ is the Darboux-transformed potential of $q_f$.
\end{corollary}
\begin{proof}
	Since $U=\frac{u}{f}\in \operatorname{Sol}_{f^2}^p(\Omega)$ and $V=fv\in \operatorname{Sol}_{\frac{1}{f^2}}^p(\Omega)$, by Lemma \ref{lemaregularidadconductivyeq}, $U,V\in W^{2,p}_{loc}(\Omega)$ and satisfy \eqref{conductivityequation1} and \eqref{conductivityequation2} a.e. in $\Omega$, respectively. Since $f\in W^{2,\infty}(\Omega)$, $u,v\in W^{2,p}_{loc}(\Omega)$. A direct computation shows that equations \eqref{conductivityequation1} and \eqref{conductivityequation2} are equivalent to \eqref{eqSchrodinger1} and \eqref{eqSchrodingerDarboux}, respectively.
\end{proof}

Denote
\begin{equation}
	\widehat{\operatorname{Sol}}_f^p(\Omega)=\left\{u\in W^{1,p}(\Omega)\, \Big{|}\, U=\frac{u}{f}\in \operatorname{Sol}_{f^2}^p(\Omega)\right\},\; \widehat{\operatorname{Sol}}_{\frac{1}{f}}^p(\Omega)=\left\{v\in W^{1,2}(\Omega)\, \Big{|}\, V=fv\in \operatorname{Sol}_{\frac{1}{f^2}}^p(\Omega)\right\}.
\end{equation}
Since the product by a proper conductivity is a bounded operation in $W^{1,p}(\Omega)$, $\widehat{\operatorname{Sol}}_f^p(\Omega)$ and $\widehat{\operatorname{Sol}}_{\frac{1}{f}}^p(\Omega)$ are closed subspaces of $W^{1,p}(\Omega)$.

In the case when $f\in C^2(\Omega)$ and $\Omega$ is a simply connected domain, in \cite[Sec. 3.4]{CamposBicomplex} it was shown that for every classical solution $u\in C^2(\Omega)$ of  \eqref{eqSchrodinger1}, a solution $v\in C^2(\Omega)$ of \eqref{eqSchrodingerDarboux} such that $W=u+\mathbf{j}v$ is a classical solution of \eqref{mainvekuaeq} (sometimes called a {\it metaharmonic conjugate} of $u$ \cite[Sec. 33]{pseudoanalyticvlad}) is given by
\begin{equation}\label{conjugateVekua}
	v=\dfrac{1}{f}\overline{A}\left[\mathbf{j}f^2\overline{\boldsymbol{\partial}}\left(\dfrac{u}{f}\right)\right]+\dfrac{c}{f}
\end{equation}
where $c$ is an arbitrary complex constant and the operator $\overline{A}$ is defined as
\begin{equation}\label{operatorA}
	\overline{A}\Phi (z)= 2\int_{\gamma}\left(\operatorname{Sc}\Phi dx+\operatorname{Vec}\Phi dy\right) =2\operatorname{Sc}\int_{\gamma}\overline{\Phi(\zeta)}d\widehat{\zeta}.
\end{equation}
The integral is taken over any curve that joins a fixed point $z_0\in \Omega$ with $z$. For the integral to be well-defined, the condition
\begin{equation}\label{irrotationalcondition}
	\frac{\partial}{\partial y}\operatorname{Sc}\Phi-\frac{\partial}{\partial x}\operatorname{Vec}\Phi =0 \quad \mbox{ in }\Omega,
\end{equation}
is required. In such a case, $\overline{\boldsymbol{\partial}}\overline{A}W=W$. We generalize this construction to weak solutions for the case when $f\in W^{1,\infty}(\Omega)$ and $\Omega$ is a star-shaped domain with respect to $z=0$. In this case, taking $z_0=0$ and $\gamma$ as the line segment joining $0$ with $z$, we introduce the operator 
\[
I_fu(z):=\overline{A}\left[\mathbf{j}f^2\overline{\boldsymbol{\partial}}\left(\dfrac{u}{f}\right)\right]=\int_0^1f^2(tz)\left(xU_y(tz)-yU_x(tz)\right)dt, \; \mbox{ with } \; U=\dfrac{u}{f}.
\]
\begin{proposition}\label{propradialoprator}
	$I_f\in \mathcal{B}(\widehat{\operatorname{Sol}}_f^p(\Omega), W^{1,p}(\Omega))$ and $\nabla I_fu=f^2(U_x,-U_y)$ for all $u\in \widehat{\operatorname{Sol}}_f^p(\Omega)$.
\end{proposition}
\begin{proof}
	Let $u\in \widehat{\operatorname{Sol}}_f^p(\Omega)$ and $U=\frac{u}{f}$. By Lemma \eqref{lemaregularidadconductivyeq}(iii), $U\in C^1(\Omega)$ and hence $I_fu(z)$ exists for every $z\in \Omega$. Note that
	\begin{align*}
			\iint_{\Omega}|I_fu(z)|^pdA_z & \leqslant \iint_{\Omega}\left[\int_{0}^1|f(tz)|^{2p}|xU_y(tz)-yU_x(tz)|^pdt\right]dA_z\\
			&\leqslant \operatorname{diam}(\Omega)^p\|f\|_{L_{\infty}(\Omega)}^{2p}\int_0^1\left[\iint_{\Omega}|\nabla U(tz)|^pdA_{z}\right]dt.
	\end{align*}
Consider the family of operators $\{B_t\}_{0\leqslant t\leqslant 1}$ given by $B_tu(z):= |\nabla U(tz)|$ for $u\in \widehat{\operatorname{Sol}}_f^p(\Omega)$. For $t>0$, $t\Omega \subset \Omega$ (because $\Omega$ is radial) and changing variables we have
\[
\|B_tu\|_{L_p(\Omega)}^2=\iint_{\Omega}|\nabla U(tz)|^p=\frac{1}{t^2}\iint_{t\Omega}|\nabla U(\zeta)|^pdA_{\zeta}\leqslant \frac{1}{t^2}\|\nabla U\|_{L_p(\Omega)}^p\leqslant \frac{1}{t^2}M_f^p\|u\|_{W^{1,p}(\Omega)}^p,
\]
with $M_f=\|f^{-1}\|_{W^{1,\infty}(\Omega)}$. For $t=0$, $B_0u(z)=|\nabla U(0)|$ (a constant function). By Lemma \ref{lemaregularidadconductivyeq}(iv), there is a constant $C_0$, independent of $U$, satisfying $|\nabla U(0)|\leqslant C_0\|U\|_{W^{1,p}(\Omega)}\leqslant C_0M_f\|u\|_{W^{1,p}(\Omega)}$. Thus, $\{B_t\}_{0\leqslant t\leqslant 1}\subset \mathcal{B}(\widehat{\operatorname{Sol}}_f^p(\Omega), L_p(\Omega))$. 

Let $\rho_0=\operatorname{dist}(0,\Gamma)$ and $\rho_1=\sup_{z\in \Omega}|z|$. Hence $\rho_0\mathbb{D}\subset \Omega \subset \rho_1 \mathbb{D}$. From the previous calculation, $\|B_tu\|_{L_p(\Omega)}^p\leqslant \left(\frac{\rho_1}{\rho_0}\right)^2M_f^p\|u\|_{W^{1,p}(\Omega)}$ for $\frac{\rho_0}{\rho_1}\leqslant t\leqslant 1$. On the other hand, for $0<t<\frac{\rho_0}{\rho_1}$, we have $t\Omega \subset t\rho_1\mathbb{D}$ and 
\[
\operatorname{Area}(t\Omega) \geqslant \operatorname{Area}(t\rho_0\mathbb{D})= \left(\frac{\rho_0}{\rho_1}\right)^2t^2\rho_1^2\operatorname{Area}(\mathbb{D})=\left(\frac{\rho_0}{\rho_1}\right)^2\operatorname{Area}(t\rho_1 \mathbb{D}),
\]
from where we obtain
\begin{align*}
	\|B_tu\|_{L_p(\Omega)}^p & = \frac{\operatorname{Area}(\Omega)}{\operatorname{Area}(t\Omega)}\iint_{t\Omega}|\nabla U(\zeta)|^pdA_{\zeta}\\
	& \leqslant \operatorname{Area}(\Omega)\left(\frac{\rho_1}{\rho_0}\right)^2\frac{1}{\operatorname{Area}(t\rho_1\mathbb{D})}\iint_{t\rho_1 \mathbb{D}}|\nabla U(\zeta)|^pdA_{\zeta}\\
	&\leqslant \operatorname{Area}(\Omega)\left(\frac{\rho_1}{\rho_0}\right)^2 \sup_{0<r<\rho_0}\frac{1}{\operatorname{Area}(r\mathbb{D})}\iint_{r \mathbb{D}}|\nabla U(\zeta)|^pdA_{\zeta}.
\end{align*}
The supremum $H_{|\nabla U|^p}(0)=\sup\limits_{0<r<\rho_0}\frac{1}{\operatorname{Area(r\mathbb{D})}}\iint_{r \mathbb{D}}|\nabla U(\zeta)|^pdA_{\zeta}$ is the {\it Hardy-Littlewood maximal function} of $|\nabla U|^p$ evaluated at $z=0$. Since $\nabla U$ is continous $\Omega$, $z=0$ is a Lebesgue point and $H_{|\nabla U|^p}(0)<\infty$ (actually, $|\nabla U(0)|^p=\lim\limits_{r\rightarrow 0^+}\frac{1}{\operatorname{Area(r\mathbb{D})}}\iint_{r \mathbb{D}}|\nabla U(\zeta)|^pdA_{\zeta}$, see \cite[Sec. 3.4]{folland}). Finally, for all $u\in \operatorname{Sol}_f(\Omega)$ we obtain the inequality
\[
\sup_{0\leqslant t\leqslant 1}\|B_tu\|_{L_p(\Omega)}^p\leqslant \left\{C_0^pM_f^p\operatorname{Area}(\Omega)\|u\|_{W^{1,p}(\Omega)}^p, \left(\frac{\rho_1}{\rho_0}\right)^2\operatorname{Area}(\Omega)H_{|\nabla U|^p}(0), \left(\frac{\rho_1}{\rho_0}\right)^2M_f^p\|u\|_{W^{1,p}(\Omega)}^p\right\}<\infty.
\]
Since $\widehat{\operatorname{Sol}}_f^p(\Omega)$ is a Banach space, by the uniform boundedness principle \cite[Th. 5.13]{folland},\\ $M_1=\sup\limits_{0\leqslant t\leqslant 1}\|B_t\|_{\mathcal{B}(\widehat{\operatorname{Sol}}_f^p(\Omega),L_p(\Omega))}<\infty$. Thus,
\[
\|I_fu\|_{L_p(\Omega)}^p \leqslant \operatorname{diam}(\Omega)^p\|f\|_{L_{\infty}(\Omega)}^{2p}\int_0^1\|B_tu\|_{L_p(\Omega)}^pdt\leqslant \operatorname{diam}(\Omega)^p\|f\|_{L_{\infty}(\Omega)}^{2p}M_1^p\|u\|_{W^{1,p}(\Omega)}^p.
\]
For the differentiability, a direct computation shows that $\phi=f^2U_y$ and $\psi=-f^2U_x$ satisfy condition \eqref{irrotationalcondition}. Taking $\varphi\in C_0^{\infty}(\Omega)$ and $G\Subset \Omega$ with $\operatorname{supp}\varphi \subset G$, we have
\begin{align*}
	\iint_{\Omega}I_fu(z)\varphi_x(z)dA_z& =\iint_{G}  \left[\int_{0}^{1}(x\phi(tz)+y\psi(tz))dt\right] \varphi_x(z)dA_z\\
	& = \int_0^1\left[\iint_{G}(x\phi(tz)+y\psi(tz)) \varphi_x(z)dA_z\right]dt.
\end{align*}	
For every $t>0$, we obtain
\begin{align*}
	\iint_{G}(x\phi(tz)+y\psi(tz)) \varphi_x(z)dA_{z}& =-\iint_{G}(\phi(tz)+tx\phi_x(tz)+ty\psi_x(tz))\varphi(z)dA_z\\
	&=-\iint_{G}(\phi(tz)+tx\phi_x(tz)+ty\phi_y(tz))\varphi(z)dA_z,
\end{align*}
where the last equality is by \eqref{irrotationalcondition}. Note that $\frac{d}{dt}(t\phi(tz))=\phi(tz)+tx\phi_x(tz)+ty\phi_y(tz)\in L_p(\Omega)$ for every $t>0$. Given $\epsilon>0$, we get
\begin{align*}
	\int_{\epsilon}^1\left[\iint_{G}(\phi(tz)+tx\phi_x(tz)+ty\phi_y(tz))\varphi(z)dA_z\right]dt& =\iint_{G}\left[\int_{\epsilon} ^1\frac{d}{dt}(t\phi(tz))dt\right]\varphi(z)dA_z\\
	& =\iint_{G}\left(\phi(z)-\epsilon \phi(\epsilon z)\right)\varphi(z)dA_z.
\end{align*}
Since $\phi\in C(\overline{G})$ (because $U_x,U_y\in C(\overline{G}))$ and $f\in C(\overline{G})$ by Remark \ref{remarksobolevembedding}), the integrant is uniformly bounded in $z$ and $\varepsilon$. Thus, by dominated convergence
\begin{align*}
	\iint_{\Omega}I_fu(z)\varphi_x(z)dA_z =-\lim\limits_{\epsilon\rightarrow 0^+}\iint_{G}\left(\phi(z)-\epsilon \phi(\epsilon z)\right)\varphi(z)dA_z=-\iint_{G}\phi(z)\varphi(z)dA_z.
\end{align*}
Therefore $\frac{\partial}{\partial x}I_fu=\phi =f^2U_x$. Similarly, $\frac{\partial}{\partial y}I_fu=\psi =-f^2U_y$ and hence $\|\nabla I_fu\|_{L_p(\Omega)}\leqslant\|f\|_{L_{\infty}(\Omega)}^2M_f\|u\|_{W^{1,p}(\Omega)}$. Consequently, $I_f\in \mathcal{B}(\widehat{\operatorname{Sol}}_f^p(\Omega), W^{1,p}(\Omega))$. 
\end{proof}

\begin{theorem}\label{theoremmetaharmonic}
Let $\Omega\subset \mathbb{C}$ be a bounded star-shaped domain with respect to $z=0$. Given $u\in \widehat{\operatorname{Sol}}_f^p(\Omega)$, 
\begin{equation}\label{metharmonicconjugate1}
	v=\frac{1}{f}I_fu\in \widehat{\operatorname{Sol}}_{\frac{1}{f}}^p(\Omega) \quad \mbox{and}\quad W=u+\mathbf{j}v\in \mathcal{A}_f^p(\Omega;\mathbb{B})\cap W^{1,p}(\Omega;\mathbb{B}).
\end{equation}
If $v_1\in W^{1,p}(\Omega)$ is another function such that $W=u+\mathbf{j}v_2\in \mathcal{A}_f^p(\Omega;\mathbb{B})\cap W^{1,2}(\Omega;\mathbb{B})$, then $v_2=\frac{1}{f}I_fu+\frac{c}{f}$, where $c\in \mathbb{C}$.
\end{theorem}
\begin{proof}
	Let $u\in \mbox{Sol}_f^p(\Omega)$. By Proposition \ref{propradialoprator}, $v=\frac{1}{f}I_fu\in W^{1,p}(\Omega)$ and $\nabla I_fu=f^2(U_y,-U_x)$. Given $\varphi\in C_0^{\infty}(\Omega)$, we get
	\begin{align*}
		\iint_{\Omega}\frac{1}{f^2}\nabla (fv)\cdot \nabla \varphi =\iint_{\Omega}\frac{1}{f^2}\nabla I_fu\cdot \nabla \varphi = \iint_{\Omega}(U_y\varphi_x-U_x\varphi_y)=-\iint_{\Omega}U(\varphi_{xy}-\varphi_{yx})=0.
	\end{align*}
Thus, $v\in \widehat{\operatorname{Sol}}_{\frac{1}{f}}^p(\Omega)$. Taking $W=u+\mathbf{j}v=fU+\mathbf{j}\frac{1}{f}I_fu\in W^{1,p}(\Omega;\mathbb{B})$, we obtain
\[
\left(\overline{\boldsymbol{\partial}}-\frac{\overline{\boldsymbol{\partial}}f}{f}\mathbf{C}_{\mathbb{B}}\right)W= f\overline{\boldsymbol{\partial}}U+\frac{\mathbf{j}}{f}\overline{\boldsymbol{\partial}}I_fu=f\overline{\boldsymbol{\partial}}U-f\overline{\boldsymbol{\partial}}U=0,
\] 
and $W\in \mathcal{A}_f^p(\Omega;\mathbb{B})$. If $v_2\in W^{1,p}(\Omega)$ is another function such that $W_2=u+\mathbf{j}v_2\in \mathcal{A}_f^p(\Omega;\mathbb{B})$, then $\mathbf{j}(v-v_2)=W-W_2\in \mathcal{A}_f^p(\Omega;\mathbb{B})$, and $0=\left(\overline{\boldsymbol{\partial}}-\frac{\overline{\boldsymbol{\partial}}f}{f}\mathbf{C}_{\mathbb{B}}\right)W=\frac{\mathbf{j}}{f}\overline{\boldsymbol{\partial}}f(v-v_2)$. This is equivalent to $\nabla f(v_2-v_2)=0$. Therefore $v_2=v+\frac{c}{f}$ for some constant $c\in \mathbb{C}$.
\end{proof}

A similar result is valid for the operator associated with $\frac{1}{f}$ and $v$.

\begin{theorem}\label{propcompletionv}
	Let $\Omega\subset \mathbb{C}$ be a star-shaped bounded domain with respect to $z=0$. The operator $I_{\frac{1}{f}}\in \mathcal{B}(\widehat{\operatorname{Sol}}_{\frac{1}{f}}^p(\Omega), W^{1,p}(\Omega))$ and for any $v\in \widehat{\operatorname{Sol}}_{\frac{1}{f}}^p(\Omega)$
	\begin{equation}\label{metharmonicconjugate2}
		u=-fI_{\frac{1}{f}}u\in \widehat{\operatorname{Sol}}_{f}^p(\Omega) \quad \mbox{and}\quad W=u+\mathbf{j}v\in \mathcal{A}_f^p(\Omega;\mathbb{B})\cap W^{1,p}(\Omega;\mathbb{B}).
	\end{equation}
If $u_2\in W^{1,p}(\Omega)$ is another function such that $W=u+\mathbf{j}v\in \mathcal{A}_f^p(\Omega;\mathbb{B})$, then $u_2=-fI_{\frac{1}{f}}v+fc$, where $c\in \mathbb{C}$.
\end{theorem}

\begin{remark}\label{remarkHilberttransform}
	Suppose that $p=2$ and $\Omega$ is a Lipschitz domain. Since $\frac{1}{f}\in L_{\infty}(\Omega)$, Eq. \eqref{conductivityequation1} is strongly elliptic in $\Omega$ \cite[p. 119]{maclean}. Let $\varphi\in W^{1,\frac{1}{2}}(\Gamma)$. There exists a unique $u\in \widehat{\operatorname{Sol}}_f^2(\Omega)$ such that $\operatorname{tr}_{\Gamma}\left(\frac{u}{f}\right)=\varphi$, and moreover, there is a constant $C_{\Omega,f}>0$, which does not depend on $\varphi$ such that $\|u\|_{W^{1,2}(\Omega)}\leqslant C_{\Omega,f}\|\varphi\|_{W^{1,\frac{1}{2}}(\Gamma)}$ (see \cite[Sec. 6.2]{evans} and \cite[Ch. 4]{maclean}). Thus, it is possible to define a composition of operators as in the following diagram 
	\[
	W^{1,\frac{1}{2}}(\Gamma) \overset{u}{\longrightarrow} \widehat{\operatorname{Sol}}_f^2(\Omega) \overset{\frac{1}{f}I_fu}{\longrightarrow}\widehat{\operatorname{Sol}}_{\frac{1}{f}}^2(\Omega)\overset{\operatorname{tr}_{\Gamma}}{\longrightarrow} W^{1,\frac{1}{2}}(\Gamma).
	\]
	Then the operator $\mathcal{H}_f:W^{1,\frac{1}{2}}(\Gamma)\rightarrow W^{1,\frac{1}{2}}(\Gamma)$ given by $\mathcal{H}_f\varphi:=\operatorname{tr}_{\Gamma}\left(\frac{1}{f}I_fu\right)$ is a type of ``Hilbert transform", in the sense that $I+\mathbf{j}\mathcal{H}_f: W^{1,\frac{1}{2}}(\Gamma)\rightarrow W^{1,\frac{1}{2}}(\Gamma;\mathbb{B})$ provides the trace of a function $W\in \mathcal{A}_f^2(\Omega;\mathbb{B})\cap W^{1,2}(\Omega;\mathbb{B})$ such that $\operatorname{tr}_{\Gamma}\operatorname{Sc}W=\varphi$.
\end{remark}

Finally, Proposition \ref{theoremmetaharmonic} allows us to extend the so-called {\it Runge property} to the main Vekua equation, for the case $p=2$.

\begin{theorem}
	Let  $\Omega_1\Subset \Omega$ be a bounded Lipschitz star-shaped domain with respect to $z=0$ such that $\Omega\setminus \overline{\Omega_1}$ is connected. Suppose that $f$ is real-valued. Then for every $\varepsilon>0$ and $W_1\in \mathcal{A}_f^2(\Omega_1;\mathbb{B})\cap W^{1,2}(\Omega_1;\mathbb{B})$, there exists $W_2\in \mathcal{A}_f^2(\Omega;\mathbb{B})\cap W^{1,2}(\Omega;\mathbb{B})$ satisfying 
	\begin{equation}\label{rungeproperty}
		\left\|W_1-W_2|_{\Omega_1}\right\|_{W^{1,2}(\Omega_1;\mathbb{B})}<\varepsilon.
	\end{equation}
\end{theorem}
\begin{proof}
	Let $W_1\in \mathcal{A}_f^2(\Omega_1;\mathbb{B})\cap W^{1,2}(\Omega_1;\mathbb{B})$ and $\varepsilon>0$. By Proposition \ref{theoremmetaharmonic}, $W=u+\frac{\mathbf{j}}{f}I_fu+\frac{c}{f}\mathbf{j}$, where $u=\operatorname{Sc}W\in \widehat{\operatorname{Sol}}_f^2(\Omega_1)$ and $c\in \mathbb{C}$. Since $f$ is a real-valued proper conductivity, Eq. \eqref{conductivityequation1} satisfies the Runge property \cite[Th. 3.9]{minerunge}. Hence, for $U_1=\frac{u}{f}\in \operatorname{Sol}_{f^2}^2(\Omega_1)$, there exists $U_2\in \operatorname{Sol}_{f^2}^2(\Omega)$ satisfying $\|U_1-U_2|_{\Omega_1}\|_{W^{1,2}(\Omega_1)}<\frac{\varepsilon}{C_f}$. where $C_f=\|f\|_{W^{1,\infty}(\Omega)}\left(1+\left\|\frac{1}{f}I_f\right\|_{\mathcal{B}(\operatorname{Sol}_f^2(\Omega_1,), W^{1,2}(\Omega_1))}\right)$. By Proposition \ref{theoremmetaharmonic}, $W_2=fU_1+\frac{\mathbf{j}}{f}I_f(fU_2)+\frac{c}{f}\mathbf{j}\in \mathcal{A}_f^2(\Omega;\mathbb{B})\cap W^{1,2}(\Omega;\mathbb{B})$. Note that $I_f\left[fU_2|_{\Omega_1}\right]=I_f[fU_2]|_{\Omega_1}$. Thus,
	\begin{align*}
		\|W_1-W_2|_{\Omega_1}\|_{W^{1,2}(\Omega_1;\mathbb{B})} &\leqslant \|fU_1-fU_2\|_{W^{1,2}(\Omega_1)}+\left\|\frac{1}{f}I_f[fU_1]-\frac{1}{f}I_f\left[fU_2|_{\Omega_1}\right]\right\|_{W^{1,2}(\Omega_1)}\\
		&\leqslant\left(\|f\|_{W^{1,\infty}(\Omega)}+\left\|\frac{1}{f}I_f\right\|_{\mathcal{B}(\operatorname{Sol}_f^2(\Omega_1,), W^{1,2}(\Omega_1))}\|f\|_{W^{1,\infty}(\Omega)}\right) \|U_1-U_2\|_{W^{1,2}(\Omega_1)}\\
		&<\varepsilon.
	\end{align*}
\end{proof}
\section{Conclusions}
We provide a definition of the Vekua-Bergman space of weak solutions of Eq. \eqref{bicomplexvekua}, along with a discussion of its completeness and the existence of a reproducing kernel. The existence of the reproducing kernel leads us to an explicit expression for the orthogonal projection onto the Vekua-Bergman space. Furthermore, an explicit expression for the orthogonal complement was obtained. For the case of the main Vekua equation, the relations between the scalar and vector parts of its solutions with the conductivity equation to the class of weak solutions were extended. Additionally, we present the construction of metaharmonic conjugates using a bounded operator in $W^{1,p}(\Omega)$. This construction enables us to establish a Runge property for the main Vekua equation, which is an important tool in constructing complete systems of solutions \cite{hugo,minerunge}. 
\section*{Acknowledgments}
The author thanks to  Instituto de Matem\'aticas de la U.N.A.M. (M\'exico), where this work was developed, and CONAHCYT for their support through the program {\it Estancias Posdoctorales por México Convocatoria 2023 (I)},


\appendix \label{appendixproofs}

\section{Appendix: Proof of some technical results}
First, we establish the continuity of the Cauchy operator $C_{\Gamma}$.

\begin{proof}[Proof of Proposition \ref{propkcauchyintegraloperator}]
	Let  $\varphi\in C^{\infty}(\overline{\Omega})$. By (\ref{BorelPompeiuformulacomplex}), $C_{\Gamma}\left[\varphi|_{\Gamma}\right](z)=\varphi(z)-A_{\Omega}\left[\frac{\partial \varphi}{\partial z^*}\right](z)$ for $z\in \Omega$. Due to Proposition \ref{proptheodorescucomplex}(ii), $\varphi-A_{\Omega}\left[\frac{\partial \varphi}{\partial z^*}\right] \in W^{1,p}(\Omega)$, and hence
	\[
	\|C_{\Gamma}\left[\varphi|_{\Gamma}\right]\|_{W^{1,p}(\Omega)}\leqslant \|\varphi\|_{W^{1,p}(\Omega)}+\left\|A_{\Omega}\left[\frac{\partial \varphi}{\partial z^*}\right]\right\|_{W^{1,p}(\Omega)}\leqslant M_1\|\varphi\|_{W^{1,p}(\Omega)},
	\] 
	with $M_1=\max\{1,\|A_{\Omega}\|_{\mathcal{B}\left(L_p(\Omega),W^{1,p}(\Omega)\right)}\}$. Since $\Omega$ is a Lipschitz domain, $\overline{C_0^{\infty}(\overline{\Omega})}^{W^{1,p}}=W^{1,p}(\Omega)$ (see \cite[Th.  3.29]{maclean}). Consequently, $C_{\Gamma} \operatorname{tr}_{\Gamma}$ can be extended to a bounded operator $\widetilde{C}_{\Gamma}: W^{1,p}(\Omega)\rightarrow W^{1,p}(\Omega)$. Let $u\in W^{1,p}(\Omega)$ and take a sequence $\{\phi_n\}\subset C^{\infty}(\overline{\Omega})$ with $\phi_n\rightarrow u$ in $W^{1,p}(\Omega)$, and $\phi_n|_{\Gamma}\rightarrow \operatorname{tr}_{\Gamma}u$ in $W^{1,1-\frac{1}{p}}(\Gamma)$. Let $v=\widetilde{C}_{\Gamma}u=\displaystyle \lim_{n\rightarrow \infty}C_{\Gamma}\left[\phi_n|_{\Gamma}\right]$ (in $W^{1,p}(\Omega)$). Since $C_{\Gamma}\left[\phi_n|_{\Gamma}\right]$ belongs to the analytic complex Bergman space $A_p(\Omega)$ and $C_{\Gamma}\left[\phi_n|_{\Gamma}\right]\rightarrow v$ in $L_p(\Omega)$, hence $C_{\Gamma}\left[\phi_n|_{\Gamma}\right]\rightarrow v$ uniformly on compact subsets of $\Omega$ \cite[Ch. 1]{duren}. Take $K\subset \Omega$ a compact, and let $\rho=\operatorname{dist}(K,\Gamma)>0$. For $z\in K$ we have
	\[
	|C_{\Gamma}[\operatorname{tr}u](z)-C_{\Gamma}\left[\phi_n|_{\Gamma}\right](z)|\leqslant \frac{1}{2\pi \rho}\int_{\Gamma}|\operatorname{tr}u(z)-\phi_n|_{\Gamma}(z)||dz|\leqslant \frac{|\Gamma|^{\frac{1}{q}}}{2\pi \rho}\|\operatorname{tr}u-\phi_n|_{\Gamma}\|_{L_p(\Gamma)}\rightarrow 0, \quad n\rightarrow \infty.
	\] 
	Thus, $C_{\Gamma}\left[\phi_n|_{\Gamma}\right]\rightarrow C_{\Gamma}\operatorname{tr}_{\Gamma}u$ uniformly on $K$. Since $K$ was arbitrary, we obtain that $C_{\Gamma}\operatorname{tr}_{\Gamma}u=v\in W^{1,p}(\Omega)$, and $ \|C_{\Gamma}\operatorname{tr}_{\Gamma}u\|_{W^{1,p}(\Omega)}= \displaystyle \lim_{n\rightarrow \infty}\|C_{\Gamma}\left[\phi_n|_{\Gamma}\right]\|_{W^{1,p}(\Omega)}\leqslant M_1\lim_{n\rightarrow \infty} \|\phi_n\|_{W^{1,p}(\Omega)}=M_1\|u\|_{W^{1,p}(\Omega)}$.
	Hence $C_{\Gamma}\operatorname{tr}_{\Gamma}\in \mathcal{B}( W^{1,p}(\Omega))$. Finally, let $\mathcal{E}:W^{1,1-\frac{1}{p}}(\Gamma)\rightarrow W^{1,p}(\Omega)$ be a bounded extension operator satisfying $\operatorname{tr}_{\Gamma}\mathcal{E}=I_{\Gamma}$ \cite[pp. 102]{maclean}. Therefore $C_{\Gamma}= \left(C_{\Gamma} \operatorname{tr}_{\Gamma}\right)\mathcal{E}\in \mathcal{B}\left(W^{1,1-\frac{1}{p}}(\Gamma),W^{1,p}(\Omega)\right)$. 
\end{proof}

Let $G\subset \mathbb{C}$ be a bounded domain. We recall that the logarithmic potential in $G$ is given by
\[
L_Gu(x)=\int_G\mathfrak{G}(x-y)u(y)dy, \quad \mbox{where }\; \mathfrak{G}(x)=
	-\frac{1}{2\pi}\log|x|.
\]
It is known that $L_Gu\in W^{2,p}(G)$ for $u\in L_p(G)$, $1<p<\infty$, and $\bigtriangleup L_Gu=u$ (see \cite[Th. 9.9]{gilbargtrudinger}). We recall that $p'=\frac{p}{p-1}$.

\begin{proof}[Proof of Lemma \ref{lemaregularidadconductivyeq}]
	\begin{itemize}
		\item[(i)] It follows from the continuity of the bilinear form
		\begin{equation}\label{apendixaux1}
			W^{1,p}(\Omega)\times W_0^{1,p'}(\Omega)\ni (u,v)\mapsto \int_G \sigma \nabla u\cdot \nabla v\in \mathbb{C}.
		\end{equation}
		\item[(ii)] Given $u\in \operatorname{Sol}_{\sigma}^p(G)$, set $h=u+L_G[\vec{\sigma}\cdot \nabla u]$, where $\vec{\sigma}=\frac{\nabla \sigma}{\sigma}$. Since $\frac{\partial u}{\partial x_i}\in L_p(\Omega)$, $i=1,\dots, N$, $L_G[\vec{\sigma}\cdot \nabla u]\in W^{1,p}(G)$. Given $\varphi\in C_0^{\infty}(G)$, we get
		\begin{align*}
			\int_G\nabla h\cdot \nabla \varphi & = \int_G \nabla u\cdot \nabla \varphi+\int_G \nabla \left(L_G[\vec{\sigma}\cdot \nabla u]\right)\cdot \nabla \varphi \\
			&= \int_G \nabla u\cdot \nabla \varphi-\int_G \left(\bigtriangleup L_G[\vec{\sigma}\cdot \nabla u]\right) \varphi \\
			&= \int_G \nabla u\cdot\left( \nabla\varphi-\varphi\frac{\nabla \sigma}{\sigma}\right)
			=\int_G\sigma \nabla u\cdot \nabla \left(\dfrac{\varphi}{\sigma}\right).
		\end{align*}
		By the continuity of the bilinear form \eqref{apendixaux1} and the density of $C_0^{\infty}(G)$ in $W_0^{1,p'}(G)$, we have that $\int_G\sigma \nabla u\cdot \nabla v=0$ for all $v\in W_0^{1,p'}(G)$. Since $\frac{1}{\sigma}\in W^{1,\infty}(G)$, $\frac{\varphi}{\sigma}\in W_0^{1,p'}(G)$ and we obtain that $\int_G \nabla h\cdot \nabla \varphi=0$ for all $\varphi\in C_0^{\infty}(G)$. Thus, $h$ is a weak solution of the Laplace equation, and by the Weyl Lemma, $h\in \operatorname{Har}(G)$. Therefore $u=h-L_G[\vec{\sigma}\cdot \nabla u]\in W^{2,p}_{loc}(G)$.

		Now, take a disk $D\Subset G$ and  write $u|_D=h_D+L_D[\vec{\sigma}\cdot \nabla u|_D]$, with $h_D\in \operatorname{Har}(D)$. Since $\frac{\partial u}{\partial x_i}\in W^{1,p}(D)$,$i=1,\dots,N$, Remark \ref{remarksobolevembedding} yields:
		
		For $1<p<2$, $\vec{\sigma}\cdot \nabla u\in L_q(D)$ for all $p\leqslant q< p^*$. Consequently $L_G[\vec{\sigma}\cdot \nabla u|_D]\in W^{2,q}(D)$. Thus, $u|_D=h_D+L_D[\vec{\sigma}\cdot \nabla u|_D]\in W_{loc}^{2,q}(D)$. Since $D$ was arbitrary, $u\in W^{2,q}_{loc}(G)$ for all $p\leqslant q<p^*$.
		
		For $p\geqslant 2$, $\vec{\sigma}\cdot \nabla u\in L_q(D)$ for $p\leqslant q<\infty$. Hence $L_D[\vec{\sigma}\cdot \nabla u|_D]\in W^{2,q}(D)$. As in the previous case, since $D$ was arbitrary, we conclude that $u\in W^{2,q}_{loc}(G)$ for all $p\leqslant q<\infty$.

		\item[(iii)] According to \cite[Th. 6 of Subsec. 5.6.3]{evans}, for every disk $D\Subset G$ and all $r>2$, $W^{2,r}(D)\hookrightarrow C^{1,1-\frac{2}{r}}(\overline{D})$. Using the fact that $p^*>2$ for $1<p<2$ and point (ii), we deduce(iii).

		\item[(iv)] Let $D\Subset G$ be a disk. Consider another disk $D'$ such that $D\Subset D'\Subset G$, and let $r>2$ be such that $u|_{D'}\in W^{2,r}(D')$ for all $u\in \operatorname{Sol}_{\sigma}^p(G)$, and the embeddings $W^{1,p}(D')\hookrightarrow L_r(D')$ and $W^{2,r}(D)\hookrightarrow C^{1,1-\frac{2}{r}}(\overline{D})$ are continuous. Let $C_{D'}^1$ and $C_{D}^1$ be their respective norms.

		Thus, for every $u\in\operatorname{Sol}_{\sigma}(G)$, $u|_{D'}\in W^{2,r}(D')$ and satisfies the conductivity equation 
		\begin{equation}\label{appendixaux2}
			\sigma \bigtriangleup u+\nabla \sigma \cdot \nabla u=\operatorname{div}\sigma \nabla u=0 \quad \mbox{ a.e. in } D'.
		\end{equation}
		Note that $\sigma \in C(\overline{D'})$ (by \cite[Th. 9.16]{brezis}), $\frac{\partial \sigma }{\partial x_i}\in L_{\infty}(D')$, $i=1,\dots, N$, and the condition $\frac{1}{\sigma}\in L_{\infty}(G)$ implies that equation \eqref{appendixaux2} is strongly elliptic in $D'$. Thus,  there exits a constant $C_{D}^2$ which does not depend on $u$ such that 
		\[
		\|u|_{D}\|_{W^{2,r}(D)}\leqslant C_{D}^2\|u|_{D'}\|_{L_r(D')}
		\]
		(see \cite[Th. 9.11]{gilbargtrudinger} and the proof of Th. 1 in Sec. 6.3.1 of \cite{evans}). Consequently,
		\begin{align*}
			\|u|_D\|_{C^1(\overline{D})}&\leqslant C_{D}^1\|u|_D\|_{W^{2,r}(D)}\leqslant C_D^1C^2_D\|u|_{B'}\|_{L_r(D')}\\
			&\leqslant C_D^1C^2_DC_{D'}^1\|u|_{D'}\|_{W^{1,p}(D')}\leqslant C_D\|u\|_{W^{1,p}(G)},
		\end{align*}
		with $C_D=C_D^1C^2_DC_{D'}^1$. For a general subset $G\Subset \Omega$, taking disks $D_1,\dots, D_M$ with $G\subset \bigcup_{j=1}^{M}D_j\Subset G$ and $C_G= \max\limits_{1\leqslant j\leqslant M}C_{D_j}$ ,we obtain \eqref{lemaregularidadconductivyeq}.
	\end{itemize}
\end{proof}
\end{document}